\documentclass[reqno,10pt,a4paper]{amsart}

\usepackage{amsfonts,amsmath,amssymb,amsthm,color}

\usepackage{bbm}
\usepackage{amssymb}
\usepackage{amsfonts}
\usepackage{amsmath}
\usepackage{amsthm}
\usepackage{nccmath}
 \usepackage{mathrsfs,dsfont}
\usepackage[USenglish]{babel}
\usepackage{esint}
\usepackage{graphicx}
\usepackage[colorlinks=true]{hyperref}
\usepackage[hmargin=3.4cm, vmargin=3.4cm]{geometry}
\hypersetup{
    colorlinks,
    citecolor=red,
    linkcolor=blue,
    urlcolor=black
}

\newtheorem{lem}{Lemma}[section]
\newtheorem{prop}{Proposition}[section]
\newtheorem{cor}{Corollary}[section]

\newtheorem{thm}{Theorem}[section]

\newtheorem*{assum}{Assumption (A)}

\theoremstyle{definition}

\theoremstyle{definition}
\newtheorem*{definition*}{Definition}

\theoremstyle{remark}

\theoremstyle{remark}
\newtheorem{remark}{Remark}[section]
\newtheorem*{remarks*}{Remarks}
\newtheorem*{remark*}{Remark}
\numberwithin{equation}{section}

\newcommand{\N}{\mathbb{N}}

\newcommand{\R}{\mathbb{R}}

\newcommand{\ov}{\overline}

\renewcommand{\ker}{\mathrm{ker}}

\newcommand{\pt}{\partial}
\newcommand{\DD}{\Delta}

\newcommand{\be}{\begin{equation}}
\newcommand{\ee}{\end{equation}}

\newcommand{\eps}{\varepsilon}

\newcommand{\KK}{\mathcal{K}}

\newcommand{\sgn}{\mathrm{sgn}}

\newcommand{\eu}{e}

\newcommand{\id}{\mathrm{Id}}

\newcommand{\Ds}{(-\Delta)^{s}}

\newcommand{\CC}{C}

\renewcommand{\phi}{\varphi}

\newcommand{\Ts}{\mathbf{T}}
\newcommand{\Tss}{\mathbf{T}^{(\sigma)}_\lambda}
\newcommand{\Fs}{\mathbf{F}}

\newcommand{\Lloc}{L_{\mathrm{loc}}}
\newcommand{\Cloc}{C_{\mathrm{loc}}}

\newcommand{\LL}{\mathcal{L}}
\newcommand{\df}{\mathrm{d}}

\catcode`@=11
\def\section{\@startsection{section}{1}%
  \z@{1.5\linespacing\@plus\linespacing}{.5\linespacing}%
  {\normalfont\bfseries\large\centering}}
\catcode`@=12

\makeatletter
\def\@cite#1#2{[\textbf{#1\if@tempswa , #2\fi}]}
\def\@biblabel#1{[\textbf{#1}]}
\makeatother

\begin{document}
\title[Fractional Gelfand Equation in $\R$]{Existence and Uniqueness for \\ the Fractional Gelfand Equation in $\R$}

\author{Florian P.~Lanz}
\address{F. P. Lanz, Universit\"at Basel, Departement Mathematik und Informatik, Spiegelgasse 1, CH-4051 Basel, Switzerland}%
\email{florian.lanz@unibas.ch}

\author{Enno Lenzmann}
\address{E. Lenzmann, University of Basel, Department of Mathematics and Computer Science, Spiegelgasse 1, CH-4051 Basel, Switzerland.}%
\email{enno.lenzmann@unibas.ch}

\begin{abstract}
We prove existence, symmetry and uniqueness of solutions to the fractional Gelfand equation
 $$
 \Ds u =  \eu^{u} \quad \mbox{in $\R$} \quad \mbox{with} \quad \int_\R  e^u \, dx < +\infty
 $$
for all exponents $s \in (\frac{1}{2},1)$. Furthermore, we show $u$ has finite Morse index and that its linearized operator is nondegenerate. 

 Our arguments are based on a fixed point scheme in terms of the function $v= \sqrt{e^u}$ and we devise a nonlocal shooting method involving (locally) compact nonlinear maps. We also study existence, symmetry and uniqueness of solutions to $\Ds u = K e^u$ in $\R$ with $K e^u \in L^1(\R)$ for a general class of positive, even and monotone-decreasing functions $K > 0$.
\end{abstract}

\maketitle

\setcounter{tocdepth}{1} 
\tableofcontents

\section{Introduction and Main Results} \label{sec:intro}

In this paper, we study solutions to the fractional Gelfand equation of the form
\be \label{eq:Gelf}
\Ds u = K e^u \quad \mbox{in $\R$} \quad \mbox{with} \quad \int_{\R} K e^u \, dx < +\infty,
\ee
where $K:\R \to \R$ is a prescribed function. Our main focus will be put on the important case of constant 
$$
K(x) \equiv 1.
$$ 
However, further results for non-constant positive $K(x) > 0$ will be discussed as well. Following standard conventions, we use $\Ds$ to denote the fractional Laplacian of order $s \in (0,1)$. For sufficiently regular and decaying functions $\phi$, we recall the singular integral expression
$$
\Ds \phi(x) = C_s \mathrm{PV} \int_{\R} \frac{\phi(x)-\phi(y)}{|x-y|^{1+2s}} \, dy = C_s \lim_{\eps \to 0^+} \int_{|x-y| \geq \eps} \frac{\phi(x)-\phi(y)}{|x-y|^{1+2s}} \, dy
$$ 
with the normalization constant $C_s = \frac{4^{s}}{\sqrt{\pi}} \frac{\Gamma(\frac{1}{2}+s)}{|\Gamma(-s)|}$. To have a mathematically meaningful notion of solutions to \eqref{eq:Gelf}, we shall always assume that $e^u \in \Lloc^1(\R)$ and $u \in L_s(\R)$, where
$$
L_s(\R) = \big \{ u \in \Lloc^1(\R) \mid \int_\R \frac{|u(x)|}{1+|x|^{1+2s}} \, dx < +\infty \big \}.
$$
We recall that $u$ is said to be a (distributional) solution to \eqref{eq:Gelf} whenever we have
$$
\int_{\R} u \Ds \phi \, dx = \int_\R e^{u} \phi \, dx 
$$
for all test functions $\phi \in \CC^\infty_c(\R)$. 

We mention that fractional Gelfand equations (in various space dimensions as well as on bounded domains) can be seen as a natural nonlocal generalization of the Liouville type equation $-\DD u = e^u$, which has been intensively studied in the area of nonlinear elliptic PDEs. We refer the reader to \cite{HyYa-22, DuNg-22, DaMa-17, RoSe-14, AhLe-22, FaWeYa-25, Ro-14, DeMaPiPr-25} for recent works on fractional Gelfand type equations in various settings.

\subsection{Existence, Symmetry, and Uniqueness for $K(x) \equiv 1$}

A particularly important case of \eqref{eq:Gelf} arises in the case of constant positive functions $K$. By a simple scaling argument, it suffices to consider the case 
$$
K(x) \equiv 1.
$$
We start with the following result about regularity and symmetry of solutions. In fact, the proofs  are straightforward adaptations of classical regularity estimates and the moving plane method. For the reader's convenience, we provide the details of the proofs in Section \ref{sec:prelim} below.
 
\begin{thm}[Smoothness and Symmetry] \label{thm:reg_sym}
Suppose $s \in (\frac 1 2, 1)$ and let $u \in L_s(\R)$ solve
\be \label{eq:Gelf2}
\Ds u = e^u \quad \mbox{with} \quad \int_\R e^u \, dx < +\infty.
\ee
Then $u$ is smooth, bounded from above and it satisfies $u(x) = O(|x|^{2s-1})$ as $|x| \to +\infty$. 

Moreover, $u$ is even and strictly monotone-decreasing around some point $x_0 \in \R$, i.\,e., we have $u(x) = u_0(|x-x_0|)$ with some function $u_0(r) < u_0(s)$ whenever $r > s \geq 0$.
\end{thm}

Next, we turn to the question of existence of solutions to \eqref{eq:Gelf2}, which constitutes the first main result in this paper.

\begin{thm}[Existence] \label{thm:existence}
For any $s \in (\frac 1 2, 1)$, there exists a solution $u \in L_s(\R)$ to \eqref{eq:Gelf2}.
\end{thm}

\begin{remark*}
Further below, we shall address the questions about the stability of solutions $u$. In particular, we see that $u$ has finite Morse index and thus it is stable outside a compact subset in $\R$. This complements the non-existence result in \cite{DuNg-22} for solutions of \eqref{eq:Gelf2} that are stable on all of $\R$; see below for more details on this issue. 
\end{remark*}

As our second main result in this paper, we  establish uniqueness of solutions for \eqref{eq:Gelf2}. Note that, by translational and scaling symmetry of this equation, we can only expect uniqueness up to the transformation
\be \label{eq:symmetry}
u(x) \mapsto u\left (\mu(x+y) \right ) + 2 s \log \mu ,
\ee 
which maps solutions into solutions for given $y \in \R$ and $\mu > 0$. The following result shows that uniqueness of solutions to \eqref{eq:Gelf2} in fact holds up to these symmetries. 

\begin{thm}[Uniqueness]  \label{thm:unique}
Let $s \in (\frac 1 2, 1)$. Then there exists a unique function $Q_s \in L_s(\R)$ such that any solution $u \in L_s(\R)$ of \eqref{eq:Gelf2} is given by 
$$
u(x) = Q_s(x)
$$ 
up to translation and scaling as in \eqref{eq:symmetry}.  
\end{thm}

\begin{remarks*}
1) By Theorem \ref{thm:reg_sym}, the function $Q_s$ is smooth, even, and strictly monotone-decreasing in $|x|$.

2) For the limiting case $s=\frac{1}{2}$, it is a well-known result  that any solution $u \in L_{1/2}(\R)$ of the one-dimensional fractional Liouville equation 
\be \tag{Liouville} \label{eq:liouville}
(-\Delta)^{\frac 1 2} u = e^u \quad \mbox{with} \quad \int_{\R} e^u \, dx < +\infty
\ee
is unique up to translation and scaling; see, e.\,g.~\cite{LiZh-95,Xu-05,DaMa-17,GeLe-24}. In fact,  we have
$$
u(x) = \log \left ( \frac{2 \mu}{1+ \mu^2 (x+y)^2} \right ) \quad \mbox{with $y \in \R$ and $\mu > 0$}.
$$
However, the case $s=\frac{1}{2}$ features conformal invariance, which is absent for \eqref{eq:Gelf2} for $s > \frac{1}{2}$. Thus our approach of proving uniqueness for $s \in (\frac{1}{2},1)$ will be based on entirely different arguments. In fact, the proof can also be adapted (with some effort) to the case $s=\frac{1}{2}$, thus providing yet another uniqueness proof in this case; see \cite{La-25} for details.

For $s=\frac{1}{2}$, we also have a striking {\em quantization phenomenon} (linked to the conformal invariance) meaning that 
$$
\int_{\R} e^u \,dx = 2 \pi
$$ 
for all solutions of \eqref{eq:liouville}. By contrast, it is easy to see that the scaling symmetry in \eqref{eq:symmetry} allows for any value $\int_{\R} e^{u} \, dx \in (0, +\infty)$ for solutions $u$ of \eqref{eq:Gelf2} in the case $s > \frac 1 2$.

3) We also mention that the limiting case $s=1$ in \eqref{eq:Gelf2} yields a nonlinear ordinary differential equation, whose solutions can be found explicitly \cite{Te-06}.
\end{remarks*}

\subsection{Finite Morse Index and Nondegeneracy for $K(x) \equiv 1$}

To complete our main results about \eqref{eq:Gelf2}, we study the Morse index of solutions as well as the nondegeneracy of the associated linearized operator $\LL_u$, which is formally given by
\be
\LL_u = \Ds - e^u .
\ee
We recall that a solution $u \in L_s(\R)$ of $\Ds u = e^u$ with the (weaker) assumption $e^u \in \Lloc^1(\R)$ is said to {\em stable outside a compact set $\KK \subset \R$} provided that
\be
\langle \phi, \LL_u \phi \rangle = \int_{\R} \phi \Ds \phi \, dx - \int_\R e^u \phi^2 \, dx \geq 0 \quad \mbox{for all $\phi \in C^\infty_c(\R \setminus \KK)$}.
\ee
That is, the quadratic form of $\LL_u$ on $C^\infty_c(\R \setminus \KK)$ is positive semi-definite. Furthermore, we recall that $u$ is {\em stable} if we can take $\KK=\emptyset$ above, i.\,e., we have that $\LL_u \geq 0$ holds in the sense of quadratic forms on $C^\infty_c(\R)$. We remark that $u$ is a stable solution if and only if its {\em Morse index}\footnote{The Morse index $n_-(\LL_u) \in \N_0 \cup \{ +\infty \}$ is the maximal dimension of the linear subspaces $V \subset C^\infty_c(\R)$ such that $\langle \phi, \LL_u \phi \rangle < 0$ for all $\phi \in V \setminus \{ 0 \}$. Clearly, we see that $u$ is stable if and only if $n_-(\LL_u) = 0$.} $n_-(\LL)$ is equal to zero. 

We recall the following non-existence result from \cite{DuNg-22}, which is inspired by the work \cite{Fa-07} on non-existence of stable solutions to the Liouville equation $-\Delta u = e^u$ in $\R^2$. 

\begin{thm}[\cite{DuNg-22}] \label{thm:non_existence}
For all $s \in (\frac{1}{2},1)$, there exists no stable solution $u \in L_s(\R)$ of $\Ds u = e^u$ with $e^u \in \Lloc^1(\R)$. 
\end{thm} 

\begin{proof}
Let $s \in (\frac 1 2, 1)$. Suppose $u \in L_s(\R)$ solves $\Ds u = e^u$ with $e^u \in \Lloc^1(\R)$ and assume that $u$ stable. By Proposition \ref{prop:morse} below, we conclude that $e^u \in L^1(\R)$ holds and thus $u$ is smooth by Theorem \ref{thm:reg_sym}. From \cite{DuNg-22} we recall that there is no stable solution $u \in C^{2s+\gamma}(\R^N) \cap L_s(\R^N)$ with $\gamma > 0$ of $(-\Delta)^s u = e^u$ in $\R^N$ if $0 < s < 1$ and $N < 10s$.
\end{proof}

The following results shows in particular that all solutions $u \in L_s(\R)$ of \eqref{eq:Gelf2} obtained by Theorem \ref{thm:existence} above have finite Morse index and are stable outside some compact non-empty set $\KK \subset \R$. 

\begin{prop} \label{prop:morse}
Let $s \in (\frac 1 2,1)$ be given. Then any solution $u \in L_s(\R)$ of \eqref{eq:Gelf2} has finite Morse index greater or equal than one, i.\,e.,
$$
1 \leq n_-(\LL_u)  < +\infty.
$$
As a consequence, the solution $u$ is stable outside some compact set $\KK \subset \R$ with $\KK \neq \emptyset$.

Conversely, if $u \in L_s(\R)$ is a solution of $\Ds u = e^u$ with $e^u \in \Lloc^1(\R)$ that is stable outside some compact set $\KK \subset \R$, then $e^u \in L^1(\R)$ holds and all conclusions of Theorems \ref{thm:reg_sym} and \ref{thm:unique} hold. In particular, $u$ has finite Morse index. 
\end{prop}

We complete the discussion of \eqref{eq:Gelf2} by showing that the linearized operator $\LL_u$ is {\em nondegenerate} in the sense that all solutions of $\LL_u \psi = 0$ are generated by the two symmetries displayed in \eqref{eq:symmetry}. To this end, we introduce the linear subspace
$$
W_s := \{ \psi \in L_s(\R) \mid \mbox{$\psi(x) = O(|x|^{2s-1})$ as $|x| \to +\infty$} \}.
$$
for $s \in (\frac{1}{2},1)$. We remark that if $u \in L_s(\R)$ solves \eqref{eq:Gelf2}, then we obtain the pointwise bound $e^{u(x)} \leq C e^{-a |x|^{2 \alpha}}$ with some constants $a > 0$ and $C > 0$. Hence it follows that $\int_{\R} e^{u(x) + \eps \psi(x)} \,dx < +\infty$ holds for $\psi \in W_s$, provided that $\eps > 0$ is sufficiently small. Thus the space $W_s$ contains admissible perturbations of \eqref{eq:Gelf2} which are consistent with the condition that $e^u \in L^1(\R)$ holds. 

We have the following result, which states that the kernel of $\LL_u$ acting on $W_s$ is entirely spanned by elements that arise from the scaling and translational symmetry of \eqref{eq:Gelf2}. That is, given a solution $u$, we consider the two-parameter family of solutions \eqref{eq:Gelf2} with
$$
u_{(y, \mu)}(x) = u \left (\mu(x+y) \right ) + 2 s \log \mu \quad \mbox{with $(y,\mu) \in \R \times \R_{>0}$},
$$
which is in accordance with \eqref{eq:symmetry}. By differentiation with respect to $y$ and $\mu$, we obtain
$$
\LL_u T_u = \LL_u R_u = 0,
$$
with the functions
$$
T_u := \pt_y u_{(0,1)}  = \pt_x u \quad \mbox{and} \quad  R_u := \pt_\mu u_{(0,1)}  = x \pt_x u + 2s.
$$
Also, we verify that $T_u(x) = o(1)$ and $R_u(x) = O(|x|^{2s-1})$ as $|x| \to \infty$, which shows that $T_u$ and $R_u$ indeed belong to  $W_s$. We now have the following nondegeneracy result.

\begin{thm}[Nondegeneracy] \label{thm:nondeg}
Let $s \in (\frac 1 2,1)$ and suppose that $u \in L_s(\R)$ solves \eqref{eq:Gelf2}. Then the corresponding linearized operator $\LL_u$ acting on $W_s$ is nondegenerate, i.\,e., 
$$
\mathrm{ker} \, \LL_u = \mathrm{span} \left \{ T_u, R_u \right \}.
$$ 
\end{thm}

\begin{remark*}
We mention the interesting fact that we only know that $n_-(\LL_u)$ is finite and we do not make use of an extra assumption such as $n_-(\LL_u) =1$. The proof of Theorem \ref{thm:nondeg} will involve a nondegeneracy result, which we derive to prove Theorem \ref{thm:unique} combined with an argument that resembles the classical {\em Birman--Schwinger trick} that occurs in the spectral study of Schr\"odinger type operators $H=(-\DD)^s  -V$. Furthermore, we also make use of some Perron--Frobenius type arguments. See below for more details on this.
\end{remark*}

\subsection{Results for Nonconstant $K(x) >0$}

For the proofs of Theorems \ref{thm:existence} and \ref{thm:unique}, we will actually need to consider the fractional Gelfand equation \eqref{eq:Gelf} for a more general class of functions $K$.  

\begin{assum}
We suppose that $K=K(x)>0$ is a strictly positive function on $\R$, which is even and monotone-decreasing in $|x|$. Moreover, we assume that $K$ is $C^1$ with $\pt_x \sqrt{K} \in L^\infty(\R)$. 
\end{assum}

\begin{remarks*}
1) Since $K$ is monotone-decreasing and positive, the assumption above implies that $K \in L^\infty(\R)$. Clearly, the following examples are admissible choices:
$$
K(x) = \mbox{const.} > 0, \quad K(x) = (1+x^2)^{-\alpha} \ \ (\alpha > 0), \quad K(x) = e^{-\beta |x|^{2m}} \ \ (\mbox{$\beta >0, m > \frac{1}{2}$}).
$$

2) The assumption $\pt_x \sqrt{K} \in L^\infty(\R)$ turns out to be convenient in our approach, but it could be relaxed at the expense of more technicalities (which we seek to avoid here). Also, we could relax the assumption that $K$ is $C^1$. Again, we refrain from doing so  in order to keep the discussion focused on the main ideas developed below.  
\end{remarks*}
 
We now summarize the following main results for \eqref{eq:Gelf}, where we exclude the case $K(x) \equiv \mbox{const}.$, which has already been  discussed above.
 
\begin{thm} \label{thm:big}
Let $s \in (\frac 1 2, 1)$ and suppose $K$ satisfies Assumption {\em \textbf{(A)}} with $K(x) \not \equiv \mathrm{const}$. Then the following statements hold true.

\begin{itemize}
\item[(i)] \textbf{Existence:} There exists an even solution $u \in L_s(\R)$ of \eqref{eq:Gelf}.
\item[(ii)] \textbf{Regularity:} Any solution $u \in L_s(\R)$ of \eqref{eq:Gelf} is in $\Cloc^{2s+\gamma}(\R)$ with some $\gamma >0$. 
\item[(iii)] \textbf{Symmetry:} Any solution $u \in L_s(\R)$ of \eqref{eq:Gelf} with $u(x) = O(|x|^{2s-1})$ as $|x| \to \infty$ is even and monotone-decreasing in $|x|$. 
\item[(iv)]  \textbf{Uniqueness:}  If $u, \tilde{u} \in L_s(\R)$ are even solutions of \eqref{eq:Gelf}, then the following implication holds:
$$
u(0) = \tilde{u}(0) \quad \Rightarrow \quad u(x) = \tilde{u}(x) \quad \mbox{for all $x \in \R$}.
$$
\item[(v)] \textbf{Finite Morse Index:} Every even solution $u \in L_s(\R)$ of \eqref{eq:Gelf} has finite Morse index.
\end{itemize}
\end{thm}

The proof of Theorem \ref{thm:big} will follow from our general analysis; see Section \ref{sec:thm_big} below for details. Notice that $u \in \Cloc^{2s + \gamma}(\R)$ implies that the equation $\Ds u = e^u$ holds in the pointwise sense in $\R$. Note also the growth condition stated in item (iii) above. Indeed, if $K$ does not decay faster than exponentially (in the sense given below), we can conclude that this growth condition for $u$ automatically holds and hence the uniqueness result is unconditional.

\begin{prop} \label{prop:main1}
Let $s$ and $K$ be as in Theorem \ref{thm:big} above and assume that $e^{\mu |x|} K \not \in L^1(\R)$ for all $\mu > 0$. Then any solution $u \in L_s(\R)$ of \eqref{eq:Gelf} satisfies $u(x) = O(|x|^{2s-1})$ as $|x| \to \infty$.
\end{prop}  

\begin{proof}
This claim follows from the more general Proposition \ref{prop:good_decay} given below.
\end{proof}
 
 \begin{remark*}
 It seems conceivable that $K$ with sufficiently fast decay (e.\,g.~take a Gaussian function $K(x) = e^{-\beta x^2}$ with $\beta > 0$) leads to existence of solutions $u \in L_s(\R)$ of \eqref{eq:Gelf} with a linear growth at infinity, i.\,e., $u(x) \sim x$ as $|x| \to +\infty$ and hence such $u$ are not even-symmetric. In particular, we may expect that uniqueness  for solutions of \eqref{eq:Gelf} may fail for such fast decaying $K(x)$. We leave this as an interesting open future problem.
 \end{remark*}
%
%

\subsection*{Acknowledgments}
Both authors gratefully acknowledge financial support from the Swiss National Science Foundation (SNSF) under Grant No.~204121. 

\section{Strategy of Proofs and Outlook}
\label{sec:strategy}

Let us summarize the arguments needed to prove the main results in this paper. For the ease of presentation, we focus on the fractional Gelfand equation with $K \equiv 1$, i.\,e., we consider   
\be \label{eq:Gelf_intro}
\Ds u = e^u \quad \mbox{with} \quad \int_{\R} e^u \, dx < +\infty,
\ee 
where $s \in (\frac 1 2, 1)$ is given. In what follows, we give a sketch of the argument whilst leaving technical points like the choice of function spaces etc.~aside. 

As a starting point, we observe that any solution $u$ of \eqref{eq:Gelf_intro} satisfies a corresponding integral equation, whose formal differentiation yields that
\be
\pt_x u = -H_\alpha(e^u).
\ee 
Here the operator $H_\alpha := H \circ (-\Delta)^{-\alpha} = H \circ |\nabla|^{-2 \alpha}$ with $\alpha = s- \frac{1}{2}$ is the {\em conjugate Riesz potential} in one space dimension of order $2\alpha \in (0,1)$ and $H$ denotes the Hilbert transform on the real line. Inspired by the recent work \cite{AhLe-22}, we now introduce the auxiliary function $v := \sqrt{e^u}$. Since this means $u = \log(v^2)$, we arrive at the equation
\be \label{eq:v_intro}
\pt_x v = -\frac{1}{2} H_\alpha(v^2) v .
\ee
In fact, this reformulation in terms $v$ turns out to be highly beneficial, as it provides us with a differential equation with a nonlocal nonlinearity that can be handled in a convenient way. Indeed, we can easily recast  \eqref{eq:v_intro} into a fixed point equation of the form
\be
v = \Ts_\lambda[v] \quad \mbox{with} \quad \Ts_\lambda[v](x) := \lambda e^{-\frac 1 2 \int_0^x H_\alpha(v^2)(y) \, dy},
\ee   
where the parameter $\lambda > 0$ corresponds to the initial-value $v(0)= \lambda$. Thanks to the symmetry result in Theorem \ref{thm:reg_sym}, it suffices to consider even solutions $v$. Hence we can work on a suitable Banach space $X_\alpha \subset L^2(\R)$ of even functions and we seek fixed points of the map $\Ts_\lambda$ on $X_\alpha$ for given $\lambda >0$. The search for fixed points of $\Ts_\lambda$ can be seen as a nonlocal ersatz of a `shooting argument' prescribing the initial-valued $v(0)=\lambda$.
 
\subsection{Existence of Fixed Points}
To prove existence of fixed points of $\Ts$, we wish to use a-priori bounds and compactness methods in the spirit of Schauder's fixed point theorem. However, a first obstruction lies in the fact that $\mathbf{T}_\lambda[0] \equiv \lambda \not \in X_\alpha \subset L^2(\R)$. More generally, it seems far from obvious how to single out a (non-empty) closed convex subset $K \subset X_\alpha$ such that $\mathbf{T}_{\lambda} : K \to K$ holds. To deal with this challenge, we introduce an auxiliary class of maps given by
\be
\Tss[v] = e^{-\frac{1}{2} \sigma^2 x^2} \Ts_\lambda[v] .
\ee
with a parameter $\sigma \in \R$. We remark that finding fixed points of $\Tss$ corresponds to looking for even solutions $u \in L_s(\R)$ of $\Ds u = K e^u$ with $K e^u \in L^1(\R)$ for the Gaussian function $K(x) = e^{-\sigma^2 x ^2}$. If $\sigma \neq 0$, the map $\Tss : X_\alpha \to X_\alpha$ is seen to be well-defined, continuous and compact. Furthermore, by rather straightforward a-priori bounds on possible fixed points, we deduce that $\Tss$ does have a fixed point in $X_\alpha$ for all $\lambda >0$ and $\sigma \neq 0$ by applying the classical Schaefer--Schauder fixed point theorem.

However, the passage to the original problem with $\mathbf{T}^{(0)}_\lambda = \mathbf{T}_\lambda$ requires a careful analysis when taking the limit $\sigma \to 0$. Here, we shall make use of Pohozaev-type identities combined with the {\em reverse Hardy--Littlewood--Sobolev (HLS) inequality}: For all $\mu > 0$ and $q \in (0,1)$ it holds that
\be \label{ineq:HLS_intro}
\int_\R \! \int_\R \rho(x) |x-y|^{\mu} \rho(y) \,dx \,dy \geq C_{\mu, q} \left ( \int_{\R} \rho(x) \,dx \right )^\beta \left ( \int_{\R} \rho(x)^q \, dx \right )^{(2-\beta)/q} 
\ee
for nonnegative $\rho \in L^1(\R) \cap L^q(\R)$, provided that $q > \frac{1}{1+\mu}$ and with a suitable choice of $\beta=\beta(q,\mu)$. Here we use  the convention that $\rho \in L^q(\R)$ for $q \in (0,1)$ if $\rho$ is measurable and $\int_\R |\rho(x)|^q \,dx$ is finite. In particular, we shall only need the reverse HLS inequality in the so-called conformally invariant case when $\beta = 0$ which is exactly the case when $q = \frac{2}{2+\mu}$. See \cite{CaDeDoFrHo-19,Be-15,NgNg-17,DoZh-15} for various proofs of \eqref{ineq:HLS_intro} and related results.

With the help of the reverse HLS inequality with $\rho = v^2$ and the fact that all fixed points are shown to be symmetric-decreasing functions, we gain sufficient control to prove existence of fixed points for the limiting problem when $\sigma=0$.

\subsection{Uniqueness and Nondegeneracy of Fixed Points}
To prove uniqueness of fixed points $v \in X_\alpha$ of $\Tss$  for given $\lambda > 0$ and $\sigma \in \R$, we make use of an implicit function argument and a global continuation argument with respect to the parameters $(\lambda, \sigma)$. As a key step in this procedure, we show that the map 
$$
\Fs(v,\lambda, \sigma) = v-\Tss[v]
$$ 
has an invertible Fr\'echet derivative $D_v \Fs$ at any fixed point $v$ of $\Tss$. By compactness of $D_v \Tss[v]$ and standard Fredholm theory, this amount to proving that $1$ is not an eigenvalue of $D_v \Tss[v]$. By ideas that are reminiscent to the monotonicity formula for the fractional Laplacian $\Ds$, we prove this spectral `nondegeneracy' property. However, in contrast to \cite{FrLe-13,FrLeSi-16}, we emphasize the fact that we do not use the $s$-harmonic extension. Instead, we work with arguments based on the Laplace transform, thereby vastly generalizing the recent approach in \cite{AhLe-22} to study the fractional Liouville equation $(-\DD)^{1/2} u = K e^u$ in $\R$ with suitable conditions on the function $K$. See also the proof of Theorem \ref{thm:nondeg}, which further elaborates on this circle of ideas and exhibits a close connection to the Birman--Schwinger principle in the spectral analysis of Schr\"odinger type operators $H=(-\DD)^s  -V$.

Thanks to the general linearized invertibility result, we can construct locally unique branches $(\lambda, \sigma) \mapsto v(\lambda, \sigma)$ of fixed points. With the help of a-priori bounds, these branches can be extended globally to all $\lambda > 0$ and $\sigma \in \R$.  On the other hand, for $\sigma \neq 0$ and $0 < \lambda \ll 1$ sufficiently small, we also see that $\Tss$ is a strict contraction and we hence obtain global uniqueness in this regime. Finally, by straightforward arguments, we can extend this to global uniqueness to all values of $\lambda >0$ and $\sigma \in \R$. In particular, this completes the general uniqueness result stated in Theorem \ref{thm:unique}, which constitutes the central result of this paper.

\subsection{Outlook on Future Work}
In an ongoing companion work \cite{LaLe-25}, we consider the corresponding parabolic problem
$$
\pt_t u + \Ds u = e^u \quad \mbox{for $t > 0$ and $x \in \R$}
$$
with $s \in [\frac 1 2,1)$ and we study stability and instability (by finite-time blowup) of the steady state solutions $u=Q_s \in L_s(\R) \cap C^\infty(\R)$ provided by Theorem \ref{thm:existence}. 

Clearly, the present work allows for many other future directions. Here a natural question would be to extend the arguments to higher space dimensions
$$
\Ds u = e^u \quad \mbox{in $\R^n$}
$$ 
with general $s > n/2$. We hope to address this problem in future work.

\section{Preliminaries} \label{sec:prelim}

We first collect some preliminaries facts whose proofs are quite elementary but worked out for the reader's convenience.

\subsection{Integral Equation and Conjugate Riesz Potential}
It will be convenient to introduce the parameter
$$
\alpha := s- \frac{1}{2} \in (0, \frac{1}{2})
$$
for  $s \in (\frac{1}{2}, 1)$ given. By straightforward arguments, we obtain the following result.

\begin{prop} \label{prop:integral_form}
Let $s \in (\frac{1}{2},1)$ and $\rho \in L^1(\R)$. Then $u \in L_s(\R)$ solves $\Ds u = \rho$ if and only if
$$
u(x) = -c_\alpha \int_{\R} \left ( |x-y|^{2 \alpha} - |y|^{2 \alpha} \right ) \rho(y) \, dy + A + B x
$$
with some constants $A, B \in \R$ and $c_\alpha = -\pi^{-1/2} 2^{-2\alpha-1} \frac{\Gamma(-\alpha)}{\Gamma(\alpha + \frac{1}{2})}>0$. 
\end{prop}

\begin{proof}
Define the function
\be \label{eq:u_def}
\tilde{u}(x) := -c_\alpha \int_\R \left ( |x-y|^{2 \alpha} - |y|^{2 \alpha} \right ) \rho(y) \, dy .
\ee
From the elementary inequality $||x-y|^{2 \alpha} - |y|^{2 \alpha}| \leq |x|^{2 \alpha}$ for $\alpha \in (0, \frac{1}{2})$, we deduce the pointwise bound $|\tilde{u}(x)| \leq c_\alpha |x|^{2 \alpha} \| \rho \|_{L^1}$ for $x \in \R$. This implies that $\tilde{u} \in L_s(\R)$. We claim that
$$
\Ds \tilde{u} = \rho \quad \mbox{in $\R$}.
$$
To show this, let $(\rho_k)_{k \geq 1} \subset C^\infty_c(\R)$ be a sequence with $\rho_k \to \rho$ in $L^1(\R)$ and define the functions $\tilde{u}_k \in L_s(\R)$ accordingly by replacing $\rho$ with $\rho_k \in C^\infty_c(\R)$ on the right side in \eqref{eq:u_def}. From the well-known fact $\Ds (G_s +\mbox{const.})= \delta$ in $\mathcal{D}'$ with the Green's function $G_s(x) = -c_\alpha |x|^{2\alpha}$, we deduce that $\Ds \tilde{u}_k = \rho_k$. Now, let $\phi \in C^\infty_c(\R)$ be given. Since $\rho_k \to \rho$ in $L^1(\R)$, we see that $\int_{\R} \rho_k \phi \to \int_{\R} \rho \phi$. On the other hand, using the well-known fact that $|\Ds \phi(x)| \leq C \langle x \rangle^{-1-2s}$ for $\phi \in C^\infty_c(\R)$ together with the bound $|\tilde{u}_k(x) - \tilde{u}(x)| \leq c_\alpha |x|^{2\alpha} \| \rho_k - \rho \|_{L^1}$, we conclude that $\int_\R w_k \Ds \phi \to \int_\R w \Ds \phi$. Hence by passing to the limit $k \to \infty$, we find that $\tilde{u} \in L_s(\R)$ solves $\Ds \tilde{u} = \rho$.

Now, assume that $u \in L_s(\R)$ solves $\Ds u = \rho$. Then $w:= u -\tilde{u} \in L_s(\R)$ solves $\Ds w = 0$ in $\R$. By known results (see, e.\,g., \cite{Fa-16}), this means that $w(x) = A + B x$ is an affine function. [Note that $A+Bx \in L_s(\R)$ for all $A, B \in \R$ when $s > \frac{1}{2}$.] This proves the claimed formula for $u$.

On the other hand, let $u$ be defined by the formula for given $\rho \in L^1(\R)$ with some constants $A,B \in \R$. Since $|u(x)| \leq c_\alpha |x|^{2 \alpha} +|A| + |Bx|$ for $x \in \R$, we conclude that $u \in L_s(\R)$ holds. From the discussion above we deduce that $u$ solves $\Ds u = \rho$.
\end{proof}

In view of Proposition \ref{prop:integral_form}, we introduce the following definition.

\begin{definition*}
Let $s \in (\frac{1}{2},1)$ and $\rho \in L^1(\R)$. We say that $u \in L_s(\R)$ is an \textbf{integral solution} of $\Ds u = \rho$ provided that
$$
u(x) = O(|x|^{2 \alpha}) \quad \mbox{as} \quad |x| \to \infty,
$$ 
i.\,e., we have that $B=0$ holds in Proposition \ref{prop:integral_form}. 

\end{definition*}

It is easy to see that even solutions $u(x)=u(-x)$ above must be necessarily integral. Also, we have the following sufficient condition with regard to solutions of fractional Gelfand-type equations in $\R$.

\begin{prop} \label{prop:good_decay}
Let $s \in (\frac{1}{2}, 1)$. Suppose that $K(x)=K(-x) \geq 0$ is an even, non-negative and measurable function on $\R$ with $e^{\mu |x|} K \not \in L^1(\R)$ for all $\mu >0$. Then any solution $u \in L_s(\R)$ of
$$
\Ds u = K e^u \quad \mbox{with} \quad \int_{\R} K e^u \, dx < +\infty
$$
must be integral. In particular, any solution $u \in L_s(\R)$ of $\Ds u = e^u$ with $\int_\R e^u \, dx < +\infty$ is integral.
\end{prop}

\begin{proof}
Define  $0 \leq \rho := K e^u \in L^1(\R)$. By Proposition \ref{prop:integral_form} and its proof, we deduce that
$$
u(x) \geq -c_\alpha |x|^{2\alpha} \| \rho \|_{L^1} + A + Bx \quad \mbox{for $x \in \R$}.
$$
with some $A, B \in \R$. Since $K \geq 0$ is even and $e^{\mu |x|} K \not \in L^1(\R)$ for all $\mu>0$, we have that $\int_{0}^\infty K(x) e^{\mu x} \,dx = +\infty$ for all $\mu > 0$. Suppose now that $B >0$ holds. Since $\alpha \in (0, \frac 1 2)$, we can find for any $0 < \eps < B$ and a constant $C=C(\eps, \| \rho \|_{L^1}) > 0$ such that $-c_\alpha \| \rho \|_{L^1} |x|^{2 \alpha} \geq -\eps x - C$ for $x \geq 0$. Therefore, we get
$$
+\infty > \int_0^\infty K(x) e^{u(x)} \, dx \geq  e^{-C+ A} \int_{0}^\infty K(x) e^{(B-\eps) x} \, dx = +\infty,
$$
which is a contradiction. The case $B < 0$ can be treated in the same way by taking integrals over $(-\infty,0]$. Thus we see that $B=0$ holds and hence $u$ is an integral solution.
\end{proof}

Next, we rewrite the integral term in Proposition \ref{prop:integral_form} using  the {\em conjugate Riesz potential operator} defined as
\be
\boxed{(H_\alpha f)(x) := d_\alpha \int_{\R} \frac{x-y}{|x-y|^{2-2 \alpha}} f(y) \, dy = d_\alpha \int_\R \frac{\mathrm{sgn}(x-y)}{|x-y|^{1-2 \alpha}} f(y) \, dy }
\ee 
for suitable functions $f$, where  $\alpha \in (0,\frac{1}{2})$ and the positive constant $d_\alpha = 2 \alpha c_\alpha > 0$ with $c_\alpha > 0$ taken from Proposition \ref{prop:integral_form} above. On suitable functions spaces, we deduce that
$$
H_\alpha = H \circ (-\Delta)^{-\alpha} .
$$
Here 
$$
H f(x) = \frac{1}{\pi} \mathrm{p.v.} \int_\R \frac{f(y)}{x-y} \, dy,  \quad  (-\Delta)^{-\alpha} f(x) = |\nabla|^{-2\alpha} f (x) = \frac{1}{\gamma_{2\alpha}} \int_\R \frac{f(y)}{|x-y|^{1-2 \alpha}} \, dy
$$ 
denote the Hilbert transform and the Riesz potential operator in one space dimension with $\alpha  \in (0,\frac{1}{2})$ and $\gamma_s = \pi^{1/2} 2^{s} \frac{\Gamma(\frac{s}{2})}{\Gamma(\frac{1-s}{2})}$, respectively. We have the following mapping properties of $H_\alpha$, which will be relevant below.

\begin{lem} \label{lem:H_a}
For any $\alpha \in (0, \frac{1}{2})$, the map $H_\alpha : L^1(\R) \cap L^\infty(\R) \to  L^\infty(\R)$ is bounded. 

Moreover, if $f \in L^1(\R) \cap L^\infty(\R)$, then the function $(H_\alpha f )(x)$ is continuous. If in addition $f=f^*$ is symmetric-decreasing, then $(H_\alpha f)(x)$ is an odd function with $(H_\alpha f)(x) \geq 0$ for $x \geq 0$ and we have strict positivity $(H_\alpha f)(x) > 0$ for all $x > 0$ whenever $f \not \equiv 0$.
\end{lem}

\begin{remarks*}
1) We refer to \cite{LiLo-01} for a general background on symmetric-decreasing rearrangements.

2) Using the weak Young inequality, we can show that $H_\alpha : L^p(\R) \to L^q(\R)$ is bounded for any $1 < p < \frac{1}{2 \alpha}$ and $1 < q < \infty$ with $\frac{1}{q} = \frac{1}{p} - 2 \alpha$. We will not need this fact here.
\end{remarks*}

\begin{proof}
We write $H_\alpha f = k_\alpha \ast f$ with $k_\alpha(x) =  \frac{\sgn(x)}{|x|^{1-2 \alpha}} \mathds{1}_{|x| \leq 1} + \frac{\sgn(x)}{|x|^{1-2 \alpha}} \mathds{1}_{|x| > 1} \in L^{p_1}(\R) + L^{p_2}(\R)$ for some $1 < p_1 < p_2  <\infty$.  Applying Young's and H\"older's inequality, it follows that 
$$
\| H_\alpha f \|_{L^\infty} \leq C \left ( \| f \|_{L^{q_1}} + \| f \|_{L^{q_2}} \right ) \leq \| f \|_{L^1 \cap L^\infty}
$$
for $q_1, q_2 \in (1, \infty)$ with $\frac{1}{p_i} + \frac{1}{q_i} =1$ for $i=1,2$. This shows that $H_\alpha : L^1(\R) \cap L^\infty(\R) \to L^\infty(\R)$ is bounded. Moreover, it is a well-known fact that $(g \ast f)(x)$ is a continuous function provided that $g \in L^p(\R)$ and $f \in L^q(\R)$ with $1 < p,q < \infty$ and $\frac{1}{p} + \frac{1}{q}=1$.

Next, we note
\be \label{eq:H_alpha}
 (H_\alpha f)(x) = d_\alpha \int_\mathbb{R} \frac{x-y}{|x-y|^{2-2\alpha}} f(y) \, dy = d_\alpha \int_0^\infty \frac{1}{y^{1-2 \alpha}} \left(f(x-y)-f(x+y)\right) dy,
\ee
 which shows that $(H_\alpha f)(x)$ is an odd function in $x$ whenever $f$ is even. Furthermore if $f=f^*$ is symmetric-decreasing, we see that $f(x-y)\geq f(x+y)$ for all $x,y\geq 0$, which implies that $(H_\alpha f)(x) \geq 0$ for $x \geq 0$. Finally, if we assume that $f \not \equiv 0$ holds. For any $x > 0$ fixed, the nonnegative function $\phi(y) = f(x-y) + f(x+y) \geq 0$ satisfies $\phi \not \equiv 0$ and we deduce from \eqref{eq:H_alpha} that $(H_\alpha f)(x) > 0$ for $x > 0$.
\end{proof}

Next, we have the following identity involving the conjugate Riesz potential in our problem.  
\begin{prop} \label{prop:riesz_integral}
For any $\rho \in L^1(\R) \cap L^\infty(\R)$, it holds that
$$
-c_\alpha \int_{\R} (|x-y|^{2 \alpha} - |y|^{2 \alpha}) \rho(y) \, dy = -\int_0^x H_\alpha(\rho)(y) \, dy \quad \mbox{for $x \in \R$}.
$$
\end{prop}

\begin{remark*}
The condition that $\rho \in L^1(\R)$ is also bounded could be relaxed, but this choice turns out to be sufficient and convenient for our purposes here. 
\end{remark*}

\begin{proof}
Let $u(x)$ denote the left-hand side of the claimed equality. We claim that
\be \label{eq:u_weak}
\pt_x u = -H_\alpha(\rho)
\ee
in the sense of weak derivatives. Indeed, let $\phi \in C^\infty_c(\R)$ be given. We readily check that we can apply Fubini's theorem to deduce that
\begin{align*}
\int_\R u(x) \pt_x \phi(x) \, dx &= -c_\alpha \int_\R \left ( \int_\R (|x-y|^{2 \alpha} - |y|^{2 \alpha}) \rho(y) \, dy \right ) \pt_x \phi(x) \, dx \\
& = -c_\alpha \int_\R \rho(y) \left ( \int_{\R} (|x-y|^{2 \alpha}- |y|^{2 \alpha}) \pt_x \phi(x) \, dx \right ) dy \\
& = 2 \alpha c_\alpha \int_{\R} \rho(y) \left ( \frac{x-y}{|x-y|^{2-2 \alpha}}  \phi(x) \, dx \right ) dy \\
& = 2 \alpha c_\alpha \left ( \int_{\R} \frac{x-y}{|x-y|^{2-2\alpha}} \rho(y) \,dy \right ) \phi(x) \, dx,
\end{align*}
where we also used that $\pt_x (|x-y|^{2 \alpha}- |y|^{2 \alpha}) = 2 \alpha \frac{x-y}{|x-y|^{2-2 \alpha}}$ in the sense of weak derivatives. This proves \eqref{eq:u_weak}. Since $H_\alpha(\rho)(x)$ is continuous for $\rho \in L^1(\R) \cap L^\infty(\R)$ by Lemma \ref{lem:H_a}, we conclude that $\pt_x u$ is of class $C^0$ and hence the claimed identity follows from integration and noticing that $u(0)=0$.
\end{proof}

\subsection{Regularity, Asymptotic Behavior, and Symmetry}

For an integer $k  \geq 0$ and a real number $\beta \in (0,1]$, we use $\Cloc^{k,\beta}(\R)$ the set of functions $u : \R \to \R$ of class $C^k$ such that $\pt^k u$ is locally H\"older continuous of order $\beta$. For a real number $\gamma > 0$, we use the standard notation $\Cloc^\gamma(\R) \equiv \Cloc^{\lfloor \gamma \rfloor, \gamma- \lfloor \gamma \rfloor}(\R)$.

\begin{lem} \label{lem:regularity}
Let $s \in (\frac 1 2, 1)$ and suppose $u \in L_s(\R)$ solves $\Ds u = \rho$ with $\rho \in L^1(\R)$. If $\rho \in \Cloc^\gamma(\R)$ with some $\gamma > 0$, then $u \in \Cloc^{\gamma+2s}(\R)$.
\end{lem}

\begin{proof}
This follows from standard Schauder type arguments. Indeed, in view of Proposition \ref{prop:integral_form} and since affine functions are smooth, it suffices to consider the case
$$
u(x) = -c_\alpha \int_\R (|x-y|^{2 \alpha}-|y|^{2 \alpha}) \rho(y) \, dy.
$$
Let $x_0 \in \R$ be given and take $j \in C^\infty_c(\R)$ with $0 \leq j \leq 1$ and $j(x) = 1$ for $x$ in an open ball $B$ centered at $x_0$. We write
$$
\rho = j \rho + (1-j) \rho =: \rho_1 + \rho_2 .
$$
Let $u_k$ with $k=1,2$ be given by the above integral expression with $\rho_k$ for $k=1,2$ on the right-hand sides. Since $\Ds u_2 = \rho_2 = 0$ in $B$, it follows that $u_2 \in C^\infty(B)$. Furthermore, we note that $u_1 = j \rho \in \Cloc^\gamma(\R)$ has compact support. Thus by following standard Schauder-type estimates (e.\,g.~take an adaption of the proof \cite{LiLo-01}[Theorem 10.3] for Poisson's equation in $\R^n$), we deduce that $u_1 \in \Cloc^{\gamma+2s}(\R)$. Therefore we deduce that $u=u_1+u_2 \in \Cloc^{\gamma+2s}(\R)$.
\end{proof}

\begin{lem} \label{lem:decay}
Let $s \in (\frac 1 2, 1)$ and $\rho \in L^1(\R)$ with $\rho \geq 0$. Suppose $u \in L_s(\R)$ is an integral solution of $\Ds u = \rho$. Then, for any $\eps \in (0,1)$, there exists $R > 0$ such that
$$
u(x) \leq -c_\alpha (1-\eps) |x|^{2 \alpha} \left ( \int_\R \rho(y) \, dy \right ) + A \quad \mbox{for} \quad  |x| \geq R
$$
with some constant $A \in \R$.
\end{lem}

\begin{proof}
Since $u \in L_s(\R)$ is an integral solution, we have 
$$
u(x) = -c_\alpha \int_{\R} (|x-y|^{2 \alpha} - |y|^{2 \alpha}) \rho(y) \,dy + A 
$$
with some constant $A \in \R$. Suppose that $\rho \not \equiv 0$, since otherwise the result follows trivially.

Let $\eps \in (0,1)$ be given and take $\delta > 0$ sufficiently small such that $|1-t|^{2 \alpha} - |t|^{2 \alpha} \geq 1-\frac{\eps}{2}$ for $|t| \leq \delta$. Furthermore, let $1 > \mu > 0$ be given and take $M> 0$ sufficiently large such that
$$
\int_{|y| > M} \rho(y) \, dy \leq \mu \| \rho \|_{L^1}
$$
Thus if we define $R:=\delta^{-1} M > 0$, we obtain $|x-y|^{2 \alpha} - |y|^{2 \alpha} \geq (1-\frac{\eps}{2}) |x|^{2 \alpha}$ for $|y| \leq M$ and $|x| \geq R$. Using that $\rho \geq 0$ and $|x-y|^{2\alpha} -|y|^{2 \alpha} \geq -|x|^{2 \alpha}$ for $\alpha \in (0, \frac{1}{2})$, we deduce
\begin{align*}
& \int_\R (|x-y|^{2 \alpha}- |y|^{2 \alpha}) \rho(y) \, dy \geq (1-\frac{\eps}{2}) |x|^{2 \alpha} \int_{|y| \leq M} \rho(y) \, dy - |x|^{2 \alpha} \int_{|y| > M} \rho(y) \, dy  \\
& \geq \left ( (1-\frac{\eps}{2}) (1-\mu) - \mu) \right ) |x|^{2 \alpha} \left ( \int_{\R} \rho(y) \,dy  \right ) \quad \mbox{for $|x| \geq R$}.
\end{align*}
Choosing now $\mu > 0$ sufficiently small such that $(1-\frac{\eps}{2})(1-\mu) - \mu \geq 1-\eps$, we obtain the claimed estimate by multiplying the above bound with $-c_\alpha < 0$.
\end{proof}

\begin{lem} \label{lem:symmetry}
Let $s \in (\frac{1}{2}, 1)$ and suppose $K \in C^1(\R)$ is even, non-negative and monotone-decreasing in $|x|$. Then any integral solution $u \in L_s(\R)$ of
$$
\Ds u = K e^u \quad \mbox{with} \quad \int_{\R} K e^u \, dx < +\infty
$$
is even and monotone-decreasing with respect to some point $x_0 \in \R$. Moreover, we have $x_0=0$ whenever $K \not \equiv \mathrm{const.}$
\end{lem}

\begin{proof}
We assume that $K \not \equiv 0$, since otherwise $u$ is an integral solution of $\Ds u = 0$ and hence $u(x) \equiv A$ for some constant $A \in \R$ and the result trivially follows.
  
Now let $\rho := K e^u \in L^1(\R)$. Since $\rho \geq 0$ with $\rho \not \equiv 0$, we deduce from Lemma \ref{lem:decay} that $u(x) \leq -a |x|^{2 \alpha} + b$ for $|x| \geq R$ with some constants $a>0$ and $b \in \R$ and some radius $R >0$. Thus $0\leq \rho(x) = K(x) e^{u(x)} \leq K(0)  e^{-a |x|^{2 \alpha}+b}$ for $|x| \geq R$ and we easily deduce that $|x|^{m} \rho \in L^1(\R)$ for any $m \geq 0$. If we choose $m=2 \alpha$, we deduce that
$$
u(x) = -c_\alpha \int_\R |x-y|^{2 \alpha} K(y) e^{u(y)} \, dy + C
$$
with some constant $C \in \R$. By applying the moving plane method detailed in Appendix \ref{sec:moving_plane} below, we finish the proof.
\end{proof}

\subsection{Proofs of Theorem \ref{thm:reg_sym}}
Let $s \in (\frac{1}{2},1)$ and suppose that $u \in L_s(\R)$ solves $\Ds u = e^u$ in $\R$ with $\int_\R e^u \, dx < +\infty$. From Proposition \ref{prop:good_decay} we deduce that $u$ is an integral solution, i.\,e., 
$$
u(x) = -c_\alpha \int_{\R} (|x-y|^{2 \alpha} - |y|^{2 \alpha}) e^{u(y)} \, dy + A
$$
with some constant $A \in \R$, whence in particular we have $u(x) = O(|x|^{2 \alpha})$ as $|x| \to +\infty$. Using the elementary estimate $| |x-y|^{2\alpha}-|z-y|^{2 \alpha}| \leq |x-z|^{2 \alpha}$ for $x,y,z \in \R$ and $2 \alpha \in (0,1)$, we obtain the global H\"older estimate
$$
|u(x)-u(z)|^{2 \alpha} \leq c_\alpha \| e^u \|_{L^1} |x-z|^{2 \alpha} \quad \mbox{for all $x,z \in \R$}.
$$
In particular, we see that $u \in \Cloc^{\gamma}(\R)$ with $\gamma = 2 \alpha = 2s -1 > 0$ and thus $e^{u} \in \Cloc^{\gamma}(\R)$. Applying Lemma \ref{lem:regularity}, we conclude that $u \in \Cloc^{\gamma + 2s}(\R)$. By iterating this argument, we find that $u \in \Cloc^{\gamma + k \cdot 2s}(\R)$ for all $k \in \N$, whence it follows that $u \in C^\infty(\R)$ is smooth. 

Next, by Lemma \ref{lem:decay}, we find for any $\eps \in (0,1)$ a radius $R > 0$ such that
$$
u(x) \leq -c_\alpha(1-\eps) |x|^{2 \alpha} \| \rho \|_{L^1} + A \quad \mbox{for $|x| \geq R$}.
$$
By this bound and the continuity of $u$, we deduce that $u(x) \leq M$ for all $x \in \R$ with some constant $M \in \R$. Thus $u$ is bounded from above, i.\,e., we have $u_+ = \max \{ u,0 \} \in L^\infty(\R)$. 

Next, we apply the moving plane method in Lemma \ref{lem:symmetry} to deduce that
$$
u(x) = u_0(|x-x_0|) 
$$
with some even function $u_0 : \R \to \R$ satisfying the monotonicity property $u(r) \leq u(s)$ for $r \geq s \geq 0$. We claim that $u_0(r)$ is indeed strictly monotone-decreasing. To see this, we note that $u_0 \in L_s(\R)$ is also an integral solution of the problem $\Ds u = e^u$ in $\R$ with $e^u \in L^1(\R)$. Since $u_0$ is bounded from above (by $u_0(0)$ due to monotonicity), we see that $\rho := e^{u_0} \in L^1(\R) \cap L^\infty(\R)$. From Proposition  \ref{prop:riesz_integral}, we obtain that
$$
\pt_x u_0 = -H_\alpha(\rho_0) .
$$
Finally, since $\rho_0 = \rho_0^*$ is symmetric-decreasing and not identically zero, we find that $H_\alpha(\rho_0)(x) > 0$ for all $x > 0$ by Lemma \ref{lem:H_a}. Thus $\pt_x u_0(x) < 0$ for all $x>0$, which implies that $u_0$ is strictly monotone-decreasing. 

This finishes the proof of Theorem \ref{thm:reg_sym}. \hfill $\qed$
\section{Setting up the Fixed Point Scheme} \label{sec:setup}

The goal of this section is to setup a fixed point scheme associated to \eqref{eq:Gelf} in terms of the function $v = \sqrt{K e^u}$. Throughout the following discussion, we fix $s \in (\frac{1}{2},1)$ and we set $\alpha = s - \frac{1}{2} \in (0, \frac{1}{2})$. We assume that the given function $K: \R \to \R_{>0}$ satisfies Assumption \textbf{(A)} above. We remind the reader that we use the notation $a \lesssim b$ to mean that $a \leq C b$ with some constant $C>0$, which may change from line to line but it only depends on $s$ (or equivalently on $\alpha$) and the function $K$. 

\subsection{Reformulation of the Problem}

We introduce the (real) Banach space
\be
X_\alpha := \{ v \in L^2(\R, |x|^{2 \alpha} \, dx) \cap L^\infty(\R) \mid \mbox{$v(x) = v(-x)$ for a.\,e.~$x \in \R$} \}
\ee
endowed the norm $\| v \|_{X^\alpha} := \| v \|_{L^2} + \| |x|^{\alpha} v \|_{L^2} + \| v \|_{L^\infty}$. Note the continuous embedding $L^2(\R, |x|^{2 \alpha} \, dx) \cap L^\infty(\R) \subset L^2(\R)$ due to the elementary estimate
$$
\| v \|_{L^2} = \| \mathds{1}_{|x| \leq 1} v \|_{L^2} + \| \mathds{1}_{|x| > 1} v \|_{L^2} \lesssim \| v \|_{L^\infty} + \| |x|^{\alpha} v \|_{L^2}.
$$
However, it will turn out to be convenient to include the term $\| v \|_{L^2}$ in our definition of the norm $\| v \|_{X^\alpha}$ henceforth.

\begin{lem} \label{lem:reformulation}
An even function $u \in L_s(\R)$ with $K e^u \in L^1(\R)$ solves the integral equation
$$
u(x) = -c_\alpha \int_\R (|x-y|^{2 \alpha} - |y|^{2 \alpha}) K(y) e^{u(y)} \, dy + u(0)
$$
with some $u(0) \in \R$ if and only if $v = \sqrt{K e^u}>0$ belongs to $X_\alpha \cap C^1$ and solves
\be \label{eq:v}
\pt_x v = \frac{1}{2} (\pt_x \log K) v - \frac{1}{2} H_\alpha(v^2) v
\ee
and $v(0) = \sqrt{K(0) e^{u(0)}}$.
\end{lem}

\begin{proof} We divide the proof as follows.

\medskip
\textbf{Step 1.} Suppose that $u \in L_s(\R)$ is an even function with $K e^u \in L^1(\R)$ that solves the integral equation. From the decay estimate in Lemma \ref{lem:decay}, we find that $u(x) \leq -a |x|^{2 \alpha} + b$ for $|x| \geq R$ with some constants $a > 0$, $b \in \R$ and $R > 0$. Since $u$ is continuous, this implies that $|e^{u(x)}| \leq C e^{-a |x|^{2 \alpha}}$ for all $x \in \R$ with some constant $C>0$. Since $K(x) > 0$ is bounded,  the even function $v := \sqrt{K e^{u}}>0$ belongs to $X_\alpha$. Furthermore, we see that $v$ is $C^1$, since $K > 0$  and $e^{u}$ are both of class $C^1$. From Proposition \ref{prop:riesz_integral}, we deduce that
$$
\pt_x u = -H_\alpha(K e^u) .
$$
Since $v^2 = K e^u$ and thus $\pt_x u = \pt_x \log(K^{-1} v^2) = 2 \frac{\pt_x v}{v} - \pt_x \log K$, we readily obtain the claimed equation satisfied by $v$. Clearly, we have that $v(0) = \sqrt{K(0) e^{u(0)}}$ holds.

\medskip
\textbf{Step 2.} Let $v \in X_\alpha \cap C^1$ be a solution of stated equation with $v(0) > 0$. By Lemma \ref{lem:ode} below, we have that
$$
v(x) = \lambda \sqrt{K(x)} e^{-\frac{1}{2} \int_0^x H_\alpha(v^2)(y) \, dy}
$$
with the constant $\lambda = v(0)/\sqrt{K(0)} > 0$. Thus we see that $v(x) > 0$ for all $x \in \R$. Hence we can define the even $C^1$-function $u := \log(K^{-1} v^2)$, which is seen to satisfy $\pt_x u = -H_\alpha(K e^u)$. Hence
$$
u(x) = -\int_0^x H_\alpha(K e^u)(y) \, dy + u(0)
$$
with $u(0) = \log(K^{-1}(0) v(0)^2)$ which is the same as $v(0) = \sqrt{K(0) e^{u(0)}}$. Since $K e^u = v^2 \in L^1(\R) \cap L^\infty(\R)$, we can apply Proposition \ref{prop:riesz_integral} to deduce that $u$ satisfies the integral equation. Finally, from the integral equation, we easily see that $|u(x)| \leq c_\alpha |x|^{2 \alpha} \| v \|_{L^2}^2 + |u(0)|$ for $x \in \R$, which  implies  $u \in L_s(\R)$.
\end{proof}

\begin{lem} \label{lem:ode}
Let $v \in X_\alpha$. Then $v$ is $C^1$ and solves \eqref{eq:v} if and only if
$$
v(x) = \lambda \sqrt{K(x)} e^{-\frac{1}{2} \int_0^x H_\alpha(v^2)(y) \, dy} 
$$
with some constant $\lambda$.
\end{lem}

\begin{proof}
Let $v \in X_\alpha$. Suppose that $v$ is given by the right-hand side of the stated equation. From Proposition \ref{prop:T_bounds} below, we find that $v$ is $C^1$ and differentiation yields that $v$ solves \eqref{eq:v}.

Assume now that $v \in X_\alpha \cap C^1$  solves \eqref{eq:v}. Using the $C^1$-function $f(x)= \frac{1}{2} \log K(x)- \frac{1}{2} \int_0^x H_\alpha(v^2)(y) \, dy$, an elementary calculation shows that $\pt_x (e^{-f} v) = 0$. Thus $e^{-f} v = \lambda$ with some constant $\lambda$ and we infer that $v$ is given by the asserted formula.
\end{proof}

\subsection{Fixed Point Setup and Compactness Properties}
In view of Lemmas \ref{lem:reformulation} and \ref{lem:ode}, we setup the following fixed point scheme. For the parameter $\lambda > 0$, we define the nonlinear map
\be \label{def:Ts}
\boxed{\Ts_\lambda[v](x) := \lambda \sqrt{K(x)} e^{-\frac{1}{2} \int_0^x H_\alpha(v^2)(y) \, dy} \quad \mbox{for $v \in X_\alpha$}}
\ee
where we recall that $H_\alpha = H \circ (-\Delta)^{-\alpha}$ denotes the conjugate Riesz potential operator; see Section \ref{sec:prelim} for elementary mapping properties. For later use, we introduce the following extension of $\Ts_\lambda$ given by
\be
\boxed{\Tss[v](x) := e^{-\frac{1}{2} \sigma^2 x^2} \Ts_\lambda[v](x)}
\ee
with the parameter $\sigma \in \R$, which is equivalent to replacing $K(x)$ with $e^{-\sigma^2 x^2} K(x)$ in the definition of $\Ts_\lambda$ above. In what follows, we shall tacitly assume that $\lambda >0$ and $\sigma \in \R$ holds, unless we explicitly state that $\sigma \neq 0$.

\begin{prop} \label{prop:T_bounds}
For all $v \in X_\alpha$, it holds
$$
0 < \Tss[v](x) \lesssim \lambda e^{-\frac{1}{2} \sigma^2 x^2} e^{-d_\alpha \| v \|_{L^2}^2 |x|^{2 \alpha} + 2 d_\alpha \| |x|^{\alpha} v \|_{L^2}^2} \quad \mbox{for $x \in \R$}
$$
with some constant $d_\alpha > 0$ depending only on $\alpha$.  Moreover, the function $x \mapsto \Tss[v](x)$ is even and $C^1$ with the pointwise bound
$$
|\pt_x \Tss[v](x)| \lesssim \lambda e^{-\frac{1}{2} \sigma^2 x^2} e^{-d_\alpha \| v \|_{L^2}^2 |x|^{2 \alpha} + 2 d_\alpha \| |x|^{\alpha} v \|_{L^2}^2} \left ( \sigma^2 |x| + \| v \|_{X^\alpha}^2 + 1 \right ) \quad \mbox{for $x \in \R$}.
$$
\end{prop}

\begin{remark}
From the pointwise bound above, we readily see that the map 
$$
\Tss : X_\alpha \to X_\alpha
$$
is well-defined provided that $\sigma \neq 0$. However, for the case $\sigma=0$, we have in general that
$$
\Ts^{(0)}_\lambda = \Ts_\lambda : X_\alpha \setminus \{ 0 \} \to X_\alpha.
$$
For instance, if $K(x) \equiv 1$, we obviously see that $\Ts_\lambda[0](x) \equiv \lambda \not \in X_\alpha$.
\end{remark}

\begin{proof}
Let $v \in X_\alpha$ be given and write $\Tss[v](x) = \lambda e^{-\frac{1}{2} \sigma^2 x^2} \sqrt{K(x)} e^{w(x)}$ with
$$
w(x) := -\frac{1}{2} \int_0^x H_\alpha(v^2)(y) \, dy = -d_\alpha \int_{\R} (|x-y|^{2 \alpha}- |y|^{2 \alpha}) v(y)^2 \, dy
$$
with some constant $d_\alpha > 0$, where the second equality is taken from Section \ref{sec:prelim}. Using the elementary inequality $-|x-y|^{2 \alpha} \leq -|x|^{2 \alpha} + |y|^{2 \alpha}$ for $\alpha \in (0, \frac{1}{2})$ together with $v^2 \geq 0$, we estimate
$$
w(x) \leq -d_\alpha \| v \|_{L^2}^2 |x|^{2 \alpha} + 2 d_\alpha \| |x|^{\alpha} v \|_{L^2}^2 \quad \mbox{for $x \in \R$}.
$$
We thus obtain the claimed pointwise bound for $0 < \Tss[v](x) \leq \lambda \| \sqrt{K} \|_{L^\infty} e^{-\frac{1}{2} \sigma^2 x^2} e^{w(x)}$.

Next, we note that $w(x)$ is even, since $v(y)^2$ is even. Hence it follows $\Tss[v](-x)=\Tss[v](x)$ by the fact $K(x)=K(-x)$. Since $K(x) > 0$ is $C^1$, we see that $\sqrt{K(x)}$ is $C^1$. Thus it remains to show that $w(x)$ is $C^1$. To see this, we note that $H_\alpha(v^2) = k_\alpha \ast v^2$ with a kernel $k_\alpha \in L^{p_1}(\R) + L^{p_2}(\R)$ with some $1 < p_1 < p_2 < \infty$. Since $v^2 \in L^1(\R) \cap L^\infty(\R)$, we deduce that $H_\alpha(v^2) = k_\alpha \ast v^2$ is a bounded and continuous functions. This shows that $w$ is $C^1$.

Finally, with $w(x)$ as above, we find that
$$
\pt_x \Tss[v](x) = \Tss[v](x) \left ( -\sigma^2 x - \frac{1}{2} H_\alpha(v^2)(x) \right ) + \lambda e^{-\frac{1}{2} \sigma^2 x^2} e^{w(x)} \pt_x \sqrt{K(x)}.
$$
Since $\| \pt_x \sqrt{K} \|_{L^\infty} \lesssim 1$ by assumption and from $\| H_\alpha(v^2) \|_{L^\infty} \lesssim \| v \|_{X^\alpha}^2$, we deduce the claimed estimate for $\pt_x \Tss[v](x)$.
\end{proof}

\begin{lem} \label{lem:Ts_compact}
The map
$$
\Tss : X_\alpha \setminus \{ 0 \} \to X_\alpha
$$
is continuous and locally compact in the sense that for every $v \in X_\alpha \setminus \{ 0 \}$ and $0 < \eps < \| v \|_{L^2}$ the image $\Tss(B_\eps(v))$ is relatively compact in $X_\alpha$.
\end{lem}

\begin{proof}
Let $\lambda > 0$ and $\sigma \in \R$ be given. We divide the proof into the following steps.

\medskip
{\em Local Compactness.} Let $v \in X_\alpha \setminus \{ 0 \}$ be given assume that $0 < \eps < \| v \|_{L^2}$. Then $B_\eps(v) \subset X_\alpha \setminus \{ 0 \}$ and $\| w \|_{L^2} \geq \| v \|_{L^2} - \eps > 0$ for any $w \in B_\eps(v)$. From Proposition \ref{prop:T_bounds} we deduce the uniform pointwise bounds
$$
|\Tss[w](x)| \leq C e^{-d |x|^{2 \alpha}} \quad \mbox{and} \quad |\pt_x \Tss[w](x)| \leq C \quad \mbox{for $x \in \R$ and $w \in B_\eps(v)$}
$$
with some constants $C>0$ and $d > 0$. Applying Lemma \ref{lem:compact} below, we conclude that the image $\Tss(B_\eps(v))$ is relatively compact in $X_\alpha$.

\medskip
{\em Continuity.} Suppose $(v_k)_{k \geq 1} \subset X_\alpha \setminus \{ 0 \}$ is a sequence such that $v_k \to v$ in $X_\alpha$ with some $v \in X_\alpha \setminus \{ 0 \}$. Since $H_\alpha(v_k^2) \to H_\alpha(v^2)$ in $L^\infty(\R)$, it follows from dominated convergence that $\int_0^x H_\alpha(v_k^2)(y) \, dy \to \int_0^x H_\alpha(v^2)(y) \, dy$ for all $x \in \R$. Hence we have the pointwise convergence
$$
\Tss[v_k](x) \to \Tss[v](x) \quad \mbox{for a.\,e.~$x \in \R$}.
$$
On the other hand, by local compactness shown above, we deduce that $(\Tss[v_{k_j}])_{j \geq 1}$ is strongly convergent in $X_\alpha$ after passing to a subsequence. But from above we deduce that the limit must be $\Tss[v](x)$ and hence independent of the chosen subsequence. This shows that $\Tss[v_k] \to \Tss[v]$ in $X_\alpha$.
\end{proof}

If $\sigma \neq 0$, we can strengthen the previous result as follows.

\begin{lem}
If $\sigma \neq 0$, then the map $\Tss : X_\alpha \to X_\alpha$ is continuous and compact.
\end{lem}

\begin{proof}
Let $(v_k)_{k \geq 1} \subset X_\alpha$ be a bounded sequence. Since $\sigma \neq 0$, we deduce from Proposition \ref{prop:T_bounds} that $| \Tss[v_k](x)| \leq C e^{-\frac{1}{2} \sigma^2 x^2} \leq C e^{-d |x|^{2 \alpha}}$ and $|\pt_x \Tss[v_k](x)| \leq C$ for all $x \in \R$ and $k \geq 1$, where $C>0$ and $d > 0$ are some constants independent of $k$. By Lemma \ref{lem:compact}, there exists a strongly convergent subsequence $(\Tss[v_{k_j}])_{j \geq 1}$ in $X_\alpha$. Thus the map $\Tss : X_\alpha \to X_\alpha$ is compact, provided that $\sigma \neq 0$.

The continuity of $\Tss : X_\alpha \to X_\alpha$ follows in analogous way as in the previous proof.
\end{proof}


\section{Existence of Fixed Points} \label{sec:apriori}

As in the previous section, we set $\alpha = s -\frac{1}{2} \in (0, \frac{1}{2})$ and we suppose that $K : \R \to \R_{>0}$ satisfies Assumption \textbf{(A)}. Unless explicitly stated otherwise, we assume that $\lambda > 0$ and $\sigma \in \R$.

\subsection{A-Priori Bounds and Existence}
We first record the important facts that fixed points of the map $\Tss$ must be symmetric-decreasing functions. For a general background on symmetric-decreasing rearrangement, we refer to \cite{LiLo-01}. 

\begin{prop}
If $v = \Tss[v]$ for some $v \in X_\alpha$, then $v=v^*$ is symmetric-decreasing.
\end{prop}

\begin{proof}
Let $v \in X_\alpha$ solve $v = \Tss[v]$ which clearly implies that $v(x) > 0$ for all $x \in \R$. From Lemma \ref{lem:reformulation} we know that $u=\log(K^{-1} v^2) \in L_s(\R)$ is an even solution of \eqref{eq:Gelf} with $K e^u \in L^1(\R)$. Since $u$ is an integral solution, we can apply Lemma \ref{lem:symmetry}  to deduce that $u$ symmetric-decreasing with respect to some point $x_0 \in \R$. Since $u$ is even, we must have $x_0=0$. Thus $v(x) = \sqrt{K(x) e^{u(x)}}$ is symmetric-decreasing, which shows that $v=v^*$. 
\end{proof}

Next, we derive a straightforward a-priori bound for fixed points in the case $\sigma \neq 0$.

\begin{lem} \label{lem:apriori_easy}
Let $\sigma \neq 0$. Then there exists a constant $M=M(\sigma, \alpha) >0$ such that, for all $v \in X_\alpha$, we have the implication
$$
\mbox{$v = \beta \Tss[v]$ for some $\beta \in [0,1]$} \quad \Rightarrow \quad \|v \|_{X^\alpha} < \lambda \|K \|_{L^\infty}^{1/2}  M . 
$$
\end{lem}

\begin{proof}
Let $\sigma \neq 0$ and $\lambda > 0$ be given. Assume that $v \in X_\alpha$ with $v = \beta \Tss[v]$ for some $\beta \in [0,1]$. Then
$$
0 \leq v(x) = \beta \lambda e^{-\frac{1}{2} \sigma^2 x^2} \sqrt{K(x)} e^{-\frac{1}{2} \int_0^x H_\alpha(v^2)(y) \, dy} \leq \lambda  \|K \|_{L^\infty}^{1/2} e^{-\frac{1}{2} \sigma^2 x^2} \quad \mbox{for $x \in \R$}.
$$
Here we used that  $\sqrt{K} \in L^\infty(\R)$ by assumption together with $H_\alpha(v^2)(y)\geq 0$ for $y \geq 0$ and $H_\alpha(v^2)(y) \leq 0$ for $y \leq 0$ by Lemma \ref{lem:H_a}, since $v=v^*$ is symmetric-decreasing. Hence $\| v \|_{X^\alpha} \leq \lambda \|K \|_{L^\infty}^{1/2} \| e^{-\frac{1}{2} \sigma^2 x^2} \|_{X^\alpha} < \lambda \|K \|_{L^\infty}^{1/2} M$ with some constant $M= M(\sigma, \alpha) > 0$.
\end{proof}

\begin{cor} \label{cor:existence_fixedpoints}
For all $\sigma \neq 0$ and $\lambda > 0$, there exists a fixed point $v \in X_\alpha$ of the map $\Tss : X_\alpha \to X_\alpha$.
\end{cor}

\begin{proof}
Since $\Tss : X_\alpha \to X_\alpha$ is continuous and compact when $\sigma \neq 0$, the result follow from the a-priori bound in Lemma \ref{lem:apriori_easy} and the classical Schaefer--Schauder fixed point theorem; see, e.\,g.~\cite{GiTr-01}[Theorem 11.3].
\end{proof}

Next, we turn to the more delicate case when $\sigma=0$ holds. Here the simple a-priori bound in Lemma \ref{lem:apriori_easy} breaks down in general, since we see that $M \to +\infty$ as $\sigma \to 0$. Moreover, we recall that the map $\Ts_{\lambda} = \Ts_\lambda^{(0)}$ is only locally compact and it is only defined on the non-convex set $X_\alpha \setminus \{ 0 \}$. Thus we need a more sophisticated analysis to deduce existence of fixed point of the map $\Ts_{\lambda}$.

Indeed, we have the following bound on fixed points for $\Tss$ independent of $\sigma \in \R$. As interesting aside, we remark that the proof involves  the {\em reverse Hardy--Littlewood--Sobolev inequality}, where the so-called conformally invariant case is sufficient for our purposes.

\begin{lem} \label{lem:HLS}
Every fixed point $v \in X_\alpha$ of $\Tss$ satisfies the bound
$$
\| v \|_{L^p} \leq C(\lambda, p, \| K \|_{L^\infty}) \quad \mbox{for} \quad \mbox{$\frac{2}{1+ \alpha} \leq p \leq \infty$},
$$
with some constant $C=C(\lambda,p,\| K \|_{L^\infty}) > 0$ independent of $\sigma \in \R$.
\end{lem}

\begin{remark}
Note that $\frac{4}{3} < \frac{2}{1+\alpha} < 2$ which follows from $0 < \alpha < \frac{1}{2}$. The fact that we obtain $L^p$-bounds for fixed points $v$ also for $p < 2$ will become useful further below, when showing existence of fixed points for the map $\mathbf{T}_\lambda=\Ts^{(0)}_\lambda$. 
\end{remark}

\begin{proof}
   Let $\sigma \in \R$  and $\lambda > 0$ be given. Suppose that $v \in X_\alpha$ solves $v = \Tss[v]$.  We recall that $v$ is  $C^1$ and solves the equation
    \begin{equation}
        \partial_x v= W v -\sigma^2 x v - \frac{1}{2} H_\alpha\left(v^2\right)v,
    \end{equation}
    where we set $W := \frac{1}{2} \pt_x \log K$. Multiplying both sides with $-xv$, followed by using that $(\partial_x v)v = \frac{1}{2} \partial_x\left(v^2\right)$ and integration over $\mathbb{R}$, we find 
    \begin{equation}
        -\int_{\mathbb{R}} x \partial_x (v^2) \, dx = -\int_{\R} W x v^2 \, dx + 2  \sigma^2 \int_{\mathbb{R}} x^2 v^2 \, dx + \int_{\mathbb{R}} x v^2 H_\alpha(v^2) \, dx.
    \end{equation}
    We remark that $|v(x)| \leq C e^{-d |x|^{2 \alpha}}$ with some constants $C, d > 0$, which is easily shows that the integrals are absolutely convergent. Next, we observe that 
    $$
    \sigma^2 \int_\mathbb{R} x^2 v^2 \, dx \geq 0 \quad \mbox{and} \quad - \! \int_\R W x v^2 \, dx \geq 0,
    $$ 
    since $-W x = -\frac{1}{2} (\pt_x \log K) x = -\frac{1}{2} \frac{\pt_x K}{K} x \geq 0$ for all $x \in \R$ due to the fact that $K=K(x) >0$ is even and monotone decreasing in $|x|$. Since $-\int_\R x \pt_x(v^2) = \int_\R v^2$, we have shown that
    \begin{equation} \label{ineq:HLS_poho}
        \int_{\mathbb{R}} v^2 \, dx \geq \int_{\mathbb{R}} x v^2 H_{\alpha}(v^2) \, dx.
    \end{equation}
    
    Let us now introduce the function $\rho := v^2 \in L^1(\mathbb{R}) \cap L^\infty(\mathbb{R})$. For the term on the right-hand side in the inequality above, we calculate (with the positive constant $d_\alpha > 0$):
    \begin{align*}
    \int_\mathbb{R} x \rho (x) H_\alpha(\rho)(x) \, dx & = d_\alpha \int_{\mathbb{R}} \! \int_{\mathbb{R}} \rho(x) x \frac{x-y}{|x-y|^{2 - 2 \alpha}} \rho(y) \, dx \, dy \\
    & =  d_\alpha \int_{\mathbb{R}} \! \int_{\mathbb{R}} \rho(x) (x-y+y) \frac{x-y}{|x-y|^{2 - 2 \alpha}} \rho(y) \, dx \, dy \\
    & = d_\alpha \int_{\mathbb{R}} \! \int_{\mathbb{R}} \rho(x) |x-y|^{2 \alpha} \rho(y) \, dx \,dy  \\
    & \quad + d_\alpha \int_{\mathbb{R}} \! \int_{\mathbb{R}} \rho(x) y \frac{x-y}{|x-y|^{2-2\alpha}} \rho(y) \,dx \,dy.
    \end{align*}
    By interchanging $x$ and $y$ in the last integral, we arrive at the identity 
    \begin{equation}
        \int_\mathbb{R} x \rho(x) H_\alpha(\rho)(x) \, dx = \frac{d_\alpha}{2} \int_\mathbb{R} \! \int_{\mathbb{R}} \rho(x) |x-y|^{2 \alpha} \rho(y) \, dx \, dy.
    \end{equation}
    In view of \eqref{ineq:HLS_poho}, we thus obtain the estimate
    \begin{equation} \label{ineq:aprior1}
        \int_{\mathbb{R}} \rho(x) \, dx \gtrsim \int_{\mathbb{R}} \! \int_{\mathbb{R}} \rho(x) |x-y|^{2 \alpha} \rho(y) \, dx \, dy. 
    \end{equation}
    To proceed further, we make use of the {\em reverse Hardy--Littlewood--Sobolev (HLS) inequality} stated in \eqref{ineq:HLS_intro} above. In fact, we shall only need the following so-called conformally invariant case of the reverse HLS-inequality, which in our case reads
  \be \label{ineq:HLS}
      \int_{\mathbb{R}} \! \int_{\mathbb{R}} \rho(x) |x-y|^{2 \alpha} \rho(y) \, dx \, dy \gtrsim \left ( \int_\mathbb{R} \rho(x)^q \, dx \right )^{2/q} \quad \mbox{with} \quad q = \frac{1}{1+\alpha} \in (0,1) .
\ee
    Next, we use that $q \in (0,1)$ and $0 \leq \rho(x) \leq \| \rho \|_{L^\infty}$ to find  
    $$
    \int_\R \left ( \frac{\rho(x)}{\| \rho \|_{L^\infty}} \right )^q \, dx \geq \int_{\R} \frac{\rho(x)}{\| \rho \|_{L^\infty}} \, dx . 
      $$ 
    Combining \eqref{ineq:aprior1} with \eqref{ineq:HLS} and recalling that $2/q= 2 + 2 \alpha$ and $1-q=\frac{\alpha}{1+\alpha}$, we conclude
    \begin{equation}
           \| \rho \|_{L^\infty}^{\frac{\alpha}{1+\alpha}} \left ( \int_{\mathbb{R}} \rho(x)^q \, dx \right ) \gtrsim  \left ( \int_\R \rho(x)^q \, dx \right )^{2+2\alpha}
    \end{equation}
    Since  $2+2 \alpha > 1$, we deduce that
    \be
    \int_{\R} \rho(x)^q \, dx \lesssim \| \rho \|_{L^\infty}^{\frac{\alpha}{1+\alpha} \cdot \frac{1}{1+2 \alpha}}.
     \ee
     Note that $\rho^q = v^{2q} = v^{\frac{2}{1+\alpha}}$ and we also have $\| \rho \|_{L^\infty} = \| v \|_{L^\infty}^2 = v(0)^2$ because $v=v^*$ is symmetric-decreasing. From the fixed point equation we see that $v(0)^2 = \lambda^2 K(0) = \lambda^2 \|K \|_{L^\infty}$, where the last equality holds since $K$ is symmetric-decreasing. In summary, we conclude hat
     \be
     \| v \|_{L^{\frac{2}{1+\alpha}}} \lesssim (\lambda \| K \|_{L^\infty}^{\frac{1}{2}})^{\frac{\alpha}{1+ 2 \alpha}} \quad \mbox{and} \quad \| v \|_{L^\infty} \leq \lambda \| K \|_{L^\infty}^{\frac{1}{2}}.
     \ee
     The claimed bound for $\| v \|_{L^p}$ now follows from simple interpolation using H\"older's inequality.
\end{proof}

\begin{lem} \label{lem:nice}
Let $\lambda > 0$ be given. Suppose that $\sigma_n \neq 0$ is a sequence such that $\sigma_n \to 0$ and let $(v_n)_{n \geq 1} \subset X_\alpha$ be a sequence of fixed points of $\Ts_{\lambda}^{(\sigma_n)}$. Then, after passing to a subsequence, it holds that $v_{n} \to v$ in $X_\alpha$ and $v \not \equiv 0$ is a fixed point of $\Ts_{\lambda}^{(0)}$.
\end{lem}

\begin{proof}
We divide the proof into the following steps.

\medskip
\textbf{Step 1.} From the uniform bounds $\| v_n \|_{L^p}$ for $p \in [\frac{2}{1+\alpha}, \infty]$ derived in Lemma \ref{lem:HLS} and the fact that the $v_n = v_n^*$ are symmetric-decreasing functions, we infer from Lemma \ref{lem:helly} (after passing to a subsequence) that
\be
\mbox{$v_n \to v$ in $L^p(\R)$ for any $\frac{2}{1+\alpha} < p < \infty$ and $v_n \to v$ almost everywhere in $\R$}
\ee
with some symmetric-decreasing function $v=v^*$ with $0 \leq v(x) \leq \lambda \sqrt{K(0)}$ for a.\,e.~$x \in \R$. We claim that
$$
e^{-\frac{1}{2} \int_0^x H_\alpha(v_n^2)(y) dy }\to e^{-\frac{1}{2} \int_0^x H_\alpha(v^2)(y) \,dy } \quad \mbox{for almost every $x \in \R$}.
$$ 
To see this, we note that $\| H_\alpha(v_n^2) \|_{L^\infty} \lesssim \| v_n^2 \|_{L^1 \cap L^\infty} = \| v_n \|_{L^2 \cap L^\infty} \lesssim 1$. Thus, by dominated convergence,
$$
\lim_{n \to \infty} \int_0^x H_\alpha(v_n^2)(y) dy = \int_0^x \lim_{n \to \infty} H_\alpha(v_n^2)(y) dy = \int_0^x H_\alpha(v^2)(y) dy
$$
for every $x \in \R$, where in the second step we use that $H_\alpha(v_n^2)(y)\to H_\alpha(v^2)(y)$ for a.\,e.~$y \in \R$, which follows from $v_n \to v$ in $L^p$ for any $p \in (2, \infty)$ and the properties of $H_\alpha$.

In the equation
\be \label{eq:vn}
v_n(x) = \lambda \sqrt{K(x)} e^{-\frac{1}{2} \sigma_n^2 x^2} e^{-\frac{1}{2} \int_0^x H_\alpha(v_n^2)(y) dy}
\ee
we can pass to the limit $n \to \infty$ on both sides of the equation to conclude that the limit $v=v^* \in L^2(\R) \cap L^\infty(\R)$ satisfies
\be \label{eq:limit}
v(x) = \lambda \sqrt{K(x)} e^{-\frac{1}{2} \int_0^x H_\alpha(v^2)(y) dy} \quad \mbox{for almost every $x \in \R$}.
\ee
In particular, this equation shows that $v \not \equiv 0$.
 
 \medskip
 \textbf{Step 2.} We demonstrate that there exist constants $C, d > 0$ such that
 \be \label{eq:v_n_bounds}
 |v_n(x)| \leq  C e^{-d |x|^{2 \alpha}}, \quad |\pt_x v_n(x)| \leq C \quad \mbox{for all $x \in \R$ and $n \geq 1$}.
 \ee
To this end, we consider the sequence of functions
$$
w_n(x) = -\frac{1}{2} \int_0^x H_\alpha(v_n^2)(y) \, dy = -c_\alpha \int_\R (|x-y|^{2 \alpha} - |y|^{2 \alpha}) v_n(y)^2 \, dy .
$$
Since $v_n \to v$ in $L^2(\R)$ with $v \not \equiv 0$, there exists a constant $c>0$ and $n_0 \geq 1$ so that
$$
\int_{\R} v_n(y)^2 \, dy \geq c \quad \mbox{for all $n \geq n_0$}.
$$
Moreover, by the uniform decay estimate $v_n(x)^2 \lesssim |x|^{-2/p}$ with $p=\frac{2}{1+\alpha} < 2$, we find that the sequence $(v_n)$ is {\em tight} in $L^2(\R)$, i.\,e., for every $\eps > 0$ there exists $M > 0$ such that
$$
\int_{|y| > M} v_n(y)^2 \, dy \leq \eps \quad \mbox{for all $n \geq 1$}.
$$ 
Next, we argue as in the proof of Lemma \ref{lem:decay}. We can take some $M>0$ such that
$$
\int_{|y| > M} v_n(y)^2 \, dy \leq \frac{c}{4} \quad \mbox{and} \quad \int_{|y| \leq M} v_n(y)^2 \, dy \geq \frac{3 c}{4} \quad \mbox{for all $n \geq n_0$}.
$$
Let $\delta > 0$ be sufficiently small such that $|1-t|^{2 \alpha} - |t|^{2 \alpha} \geq \frac{1}{2}$ for $|t| \leq \delta$. Thus if we define $R:=\delta^{-1} M > 0$, we obtain
$$
|x-y|^{2 \alpha} - |y|^{2 \alpha} \geq \frac{1}{2} |x|^{2 \alpha} \quad \mbox{for $|y| \leq M$ and $|x| \geq R$}.
$$
Using that $v_n^2 \geq 0$ and the general bound $|x-y|^{2\alpha} -|y|^{2 \alpha} \geq -|x|^{2 \alpha}$, we deduce
\begin{align*}
& \int_\R (|x-y|^{2 \alpha}- |y|^{2 \alpha}) v_n(y)^2 \, dy \geq \frac{1}{2} |x|^{2 \alpha} \int_{|y| \leq M} v_n(y)^2 \, dy - |x|^{2 \alpha} \int_{|y| > M} v_n(y)^2 \, dy \\
& \geq \frac{3 c}{8} |x|^{2 \alpha} - \frac{c}{4} |x|^{2 \alpha} = \frac{c}{8} |x|^{2 \alpha} \quad \mbox{for $|x| \geq R$}.
\end{align*}
Choosing now $d = c_\alpha \cdot \frac{c}{8} >0$, we obtain the bound
$$
w_n(x) \leq -d |x|^{2 \alpha} \quad \mbox{for $|x| \geq R$ and $n \geq n_0$}.
$$
Therefore,
$$
0 < v_n(x) = \lambda \sqrt{K(x)} e^{-\frac{1}{2} \sigma_n^2 x^2} e^{w_n(x)}  \leq \lambda \| K \|_{L^\infty}^{1/2}  e^{-d |x|^{2 \alpha}} \quad \mbox{for $|x| \geq R$ and $n \geq n_0$}.
$$

Since also $0 < v_n(x) \leq v_n(0) = \lambda \sqrt{K(0)} = \lambda \| K \|_{L^\infty}^{1/2}$ for all $x \in \R$, we readily deduce that the first estimate in \eqref{eq:v_n_bounds} holds for some sufficiently large constant $C>0$ independent of $n$. Also, from the general estimate in Proposition \ref{prop:T_bounds}, we deduce
$$
|\pt_x v_n(x)| \leq C e^{-d|x|^{2 \alpha}} (\sigma_n^2 |x| + 1) \leq C \quad \mbox{for all $x \in \R$ and $n \geq 1$}
$$
with some constant $C >0$ independent of $n$. This proves \eqref{eq:v_n_bounds}.

By Lemma \ref{lem:compact}, we infer (after passing to a subsequence) that $v_n \to v$ in $X_\alpha$ and $v \not \equiv 0$ is a fixed point of $\Ts_\lambda=\Ts_{\lambda}^{(0)}$ thanks to \eqref{eq:limit}. The proof of Lemma \ref{lem:nice} is now complete.
\end{proof}

By Corollary \ref{cor:existence_fixedpoints}, we have existence of fixed points of $\Tss$ for all $\sigma \neq 0$. In view of Lemma \ref{lem:nice}, we obtain the following existence result.

\begin{cor} \label{cor:exist}
For any $\lambda > 0$, the map $\Ts_{\lambda} = \Ts_{\lambda}^{(0)} : X_\alpha \setminus \{ 0 \} \to X_\alpha$ has a fixed point.
\end{cor} 

\subsection{Proof of Theorem \ref{thm:existence}} 
Take $K(x) \equiv 1$ and let $\lambda > 0$ be given. By Corollary \ref{cor:exist}, there exists $v \in X_\alpha \setminus \{ 0 \}$ such that $v = \mathbf{T}_{\lambda}[v]$. Hence $v(x) > 0$ is a $C^1$-solution of 
$$
\pt_x v = -\frac{1}{2} H_\alpha(v^2) v
$$
with $v(0) = \lambda \sqrt{K(0)} = \lambda$. In view of Lemma \ref{lem:reformulation} and Proposition \ref{prop:integral_form}, we obtain that the even function $u := \log(v^2) \in L_s(\R)$ solves $\Ds u = e^u$ with $e^u = v^2 \in L^1(\R)$. \hfill $\qed$

\section{Uniqueness of Fixed Points} \label{sec:unique}

We now turn to the uniqueness of fixed points of $\Tss$ for given $\lambda >0$ and $\sigma \in \R$. Again, we start with the easier case when $\sigma \neq 0$ holds and we first show uniqueness in the regime of small $\lambda >0$.

\begin{lem} \label{lem:unique_small}
Let $\sigma \neq 0$ be given. Then there exists $\lambda_0 = \lambda_0(\sigma, \alpha, \|K \|_{L^\infty}) > 0$ sufficiently small such that the fixed points of $\Tss : X_\alpha \to X_\alpha$ are unique for $\lambda \in (0, \lambda_0]$. 
\end{lem}

\begin{proof}
Let $\sigma \neq 0$ be fixed. In Section \ref{sec:aux} we prove that  $\Tss : X_\alpha \to X_\alpha$ is Fr\'echet differentiable and that its derivative $D_v \Tss[v] \in \mathcal{L}(X_\alpha)$ depends continuously on $(\lambda, v) \in (0,\infty) \times X_\alpha$. Since $D_v \Tss[0] = 0$ and by continuity, we deduce that
$$
\| D_v \Tss[v] \|_{\mathcal{L}(X_\alpha)} \leq \frac{1}{2} \quad \mbox{if} \quad  \mbox{$\| v \|_{X^\alpha} \leq r_0$ and $0 < \lambda \leq \lambda_0$}
$$
with some sufficiently small constants $r_0 > 0$ and $\lambda_0 > 0$.
From the mean-value estimate due to Taylor's theorem in Banach spaces, we deduce
\be \label{ineq:contraction}
\| \Tss[v] - \Tss[\tilde{v}] \|_{X^\alpha} \leq \frac{1}{2} \|v - \tilde{v} \|_{X^\alpha} \quad \mbox{if} \quad \mbox{$\| v \|_{X^\alpha}, \| \tilde{v} \|_{X^\alpha} \leq r_0$ and $0 < \lambda \leq \lambda_0$}.
\ee

Now, let $M=M(\sigma, \alpha) > 0$ be the constant from Lemma \ref{lem:apriori_easy}. By choosing $\lambda_0 > 0$ even smaller, we can ensure that $\lambda_0 \|K \|_{L^\infty}^{1/2} M \leq r_0$ holds. Thus if $\lambda \in (0, \lambda_0]$ and $v, \tilde{v} \in X_\alpha$ are fixed points of $\Tss$, we deduce that $\| v \|_{X^\alpha} \leq r_0$ and $\| \tilde{v} \|_{X^\alpha} \leq r_0$ by Lemma \ref{lem:apriori_easy}. The contraction estimate \eqref{ineq:contraction} now implies that $v = \tilde{v}$.
\end{proof}

Our next goal is to extend the previous uniqueness result to all $\lambda > 0$ by means of a global continuation argument using the implicit function theorem. To this end, we introduce the following map
$$
\Fs : (0,\infty) \times \R \times (X_\alpha \setminus \{ 0 \}) \to X_\alpha, \quad (\lambda, \sigma, v) \mapsto \Fs(\lambda, \sigma, v) := v - \Tss[v].
$$
Obviously, we have the equivalence
$$
\Fs(\lambda, \sigma, v)=0 \quad \Longleftrightarrow \quad v= \Tss[v].
$$
From the known properties of $\Tss$, we directly deduce that $\Fs$ is Fr\'echet differentiable with respect to $v$ and its derivative is given by
$$
D_v \Fs(\lambda, \sigma, v) = \id - D_v \Tss[v]
$$
and it depends continuously on $(\lambda, \sigma, v)$. We note that $D_v \Fs(\lambda, \sigma, v) : X_\alpha \to X_\alpha$ is Fredholm, since $D_v \Tss[v]$ is a compact linear operator. The following general nondegeneracy result for the linearization around fixed points is essential.

\begin{lem}[Nondegeneracy of Fixed Points] \label{lem:nondeg}
Let $\sigma \in \R$ and $\lambda > 0$. Suppose that $v \in X_\alpha$ satisfies $\Fs(\lambda, \sigma, v) = 0$. Then $D_v \Fs(\lambda, \sigma, v) : X_\alpha \to X_\alpha$ is invertible. 
\end{lem}

\begin{proof}
Let $(\lambda, \sigma, v)$ be as above. By standard Fredholm theory, we know that $D_v \Fs(\lambda, \sigma, v)$ is invertible if (and only if) the compact operator $\KK := D_v \Tss[v] \in \mathcal{L}(X_\alpha)$ satisfies
\be \label{eq:ker}
\ker (\id - \KK ) = \{ 0 \}.
\ee
To show this, we argue by contradiction and we suppose that there exists some $f \in X_\alpha \setminus \{ 0 \}$ with  $\KK f = f$. Since $v = \Tss[v]$ is a fixed point, we deduce from Proposition \ref{prop:frechet} that $f$ satisfies the equation
\be \label{eq:f_eq}
f(x) = -v(x) \int_0^x H_\alpha(vf)(y) \, dy.
\ee
We claim that this identity will imply that $f \equiv 0$, completing the proof of \eqref{eq:ker}.

Indeed, let us consider the quantity
\be
I[f] := \int_0^\infty (vf)(x) H_\alpha(vf)(x) \, dx .
\ee
Note that this integral is absolutely convergent with
$$
|I[f]| \lesssim \int_0^\infty \int_{\R} \frac{|v(x)f(x)||v(y) f(y)|}{|x-y|^{1-2 \alpha}} \, dx \, dy \lesssim \| v f \|_{L^{\frac{2}{1+2 \alpha}}}^2 < +\infty,
$$
using the Hardy--Littlewood--Sobolev inequality and that $vf \in L^p(\R)$ for all $1 \leq p \leq \infty$ whenever $f,v \in X_\alpha$. We claim that 
\be \label{ineq:I_one}
I[f] \geq 0 \quad \mbox{with} \quad \mbox{$I[f]=0$ if and only if $f =0$}.
\ee 
To see this, let us set 
$$
w:= vf.
$$ 
From the explicit form of the integral kernel of $H_\alpha$ we find (with some positive constant $d_\alpha > 0$) that
\begin{align*}
I[f] & = d_\alpha \int_0^\infty w(x) \left ( \int_{\R} \frac{\sgn(x-y)}{|x-y|^{1-2 \alpha}} w(y) \, dy \right ) dx \\
& = d_\alpha \int_0^\infty w(x) \left ( \int_{-\infty}^0 \frac{\sgn(x-y)}{|x-y|^{1-2 \alpha}} w(y) \, dy + \int_0^\infty \frac{\sgn(x-y)}{|x-y|^{1-2 \alpha}} w(y) \, dy \right )  dx \\
& = d_\alpha \int_0^\infty w(x) \int_{-\infty}^0 \frac{\sgn(x-y)}{|x-y|^{1-2 \alpha}} w(y) \, dy \, dx = d_\alpha \int_0^\infty \int_0^\infty \frac{w(x) w(y)}{(x+y)^{1-2 \alpha}} \, dx \, dy.
\end{align*}
Here we used that the second integral in the second line vanishes by skew-symmetry. In the last step, we used that $w(-y) = w(y)$ is an even function. Next, we recall that
$$
\frac{1}{(x+y)^{1-2 \alpha}} = \frac{1}{\Gamma(2 \alpha+1)} \int_0^\infty t^{2 \alpha} e^{-t(x+y)} \, dt
$$
for $x,y >0$. This leads us to
$$
I[f] = d_\alpha \int_0^\infty \int_0^\infty \frac{w(x) w(y)}{(x+y)^{1-2 \alpha}} \, dx \, dy = \frac{d_\alpha}{\Gamma(2 \alpha +1)} \int_0^\infty t^{2 \alpha} |\mathsf{L}w(t)|^2 \, dt \geq 0,
$$
where $\mathsf{L} w(t) = \int_0^\infty e^{-t s} w(s) \, ds$ denotes the (one-sided) Laplace transform of $w : [0,\infty) \to \R$. This shows that $I[f] \geq 0$ and clearly $I[0]=0$. Moreover, if $I[f]=0$ then $\mathsf{L}w(t)=0$ for a.\,e.~$t \geq 0$. Since $w=vf : [0,\infty) \to \R$ is easily seen to be continuous and bounded, we can apply well-known arguments for the Laplace transform (stated in Lemma \ref{lem:Laplace} below) to conclude that $w=vf=0$. By the strict positivity $v(x) > 0$, this implies that $f=0$.  This completes the proof of \eqref{ineq:I_one}. 

On the other hand, we claim that
\be \label{ineq:I_two}
I[f] \leq 0.
\ee
To see this, we mutliply \eqref{eq:f_eq} with $v(x) > 0$ yielding
$$
v(x) f(x) = v(x)^2 \psi(x),
$$
where we define the $C^1$-function
$$
\psi(x) = -\int_0^x H_\alpha(vf)(y) \, dy,
$$
which satisfies the estimate $|\psi(x)| \leq |x| \| H_\alpha(vf) \|_{L^\infty} \lesssim |x| \| v f \|_{L^1 \cap L^\infty} \lesssim |x| \| v \|_{X^\alpha} \| f \|_{X^\alpha}$ for all $x \in \R$. We have
\begin{align*}
I[f] & = \int_0^\infty (v f)(x) H_\alpha(vf)(x) \, dx = -\int_0^\infty v(x)^2 \psi(x) \pt_x \psi(x) \, dx = -\frac{1}{2} \int_0^\infty v(x)^2 \pt_x \left \{ \psi(x)^2 \right\} \, dx \\
& = -\frac{1}{2} v(0)^2 \psi(0)^2 + \frac{1}{2} \int_0^\infty \pt_x \{ v(x)^2 \} \psi(x)^2 \, dx = \frac{1}{2} \int_0^\infty \pt_x \left \{ v(x)^2 \right \} \psi(x)^2 \, dx,
\end{align*}
where we integrated by parts using that $\psi(0)=0$ and that $v(x)^2 \psi(x) \to 0$ as $x \to +\infty$ due to $v(x) \leq C e^{-d|x|^{2 \alpha}}$ with some $d>0$ and the fact that $\psi(x) = O(x)$ as $|x| \to +\infty$. Since $v=v^*$ is a symmetric-decreasing $C^1$-function, we see that $\pt_x(v(x)^2) \leq 0$ for all $x \geq 0$. Hence we conclude that \eqref{ineq:I_two} holds.

In view of \eqref{ineq:I_one} and \eqref{ineq:I_two}, we deduce that $\mathcal{K} f = f$ implies that $I[f]=0$ and hence $f = 0$. This proves \eqref{eq:ker} and completes the proof of Lemma \ref{lem:nondeg}.
\end{proof}

We are now ready to prove the following uniqueness result for fixed points of  $\Tss$.

\begin{lem} \label{lem:unique}
For all $\sigma \in \R$ and $\lambda >0$, we have uniqueness of fixed points for $\Tss$ in $X_\alpha$.
\end{lem}

\begin{proof}
We discuss the cases $\sigma \neq 0$ and $\sigma=0$ separately as follows.

\medskip
\textbf{Case $\sigma \neq 0$.} Assume that $\sigma \neq 0$ is fixed and let $\lambda_* > 0$ be given. We argue by contradiction and we suppose that $v_*, \tilde{v}_* \in X_\alpha$ are two fixed points of $\Ts_{\lambda_*}^{(\sigma)}$ with $v_* \neq \tilde{v}_*$. Since $\Fs(\lambda_*, \sigma, v_*)=\Fs(\lambda_*, \sigma, \tilde{v}_*)=0$, we can apply the implicit function theorem (using the nondegeneracy result in Lemma \ref{lem:nondeg}) together with the a-priori bound in Lemma \ref{lem:apriori_easy} to conclude that there exist two continuous branches $v(\lambda), \tilde{v}(\lambda) \in C((0, \lambda_*]; X_\alpha)$ with
$$
v(\lambda) = \Ts_{\lambda}^{(\sigma)}[v(\lambda)] \quad \mbox{and} \quad \tilde{v}(\lambda) = \Ts_{\lambda}^{(\sigma)}[\tilde{v}(\lambda)] \quad \mbox{for $\lambda \in (0, \lambda_*]$}
$$
with $v(\lambda_*) = v_*$ and $\tilde{v}(\lambda_*) = \tilde{v}_*$. By local uniqueness thanks to the implicit function theorem, these branches can never intersect, i.\,e., we have $v(\lambda) \neq \tilde{v}(\lambda)$ for $\lambda \in (0, \lambda_*]$. However, by Lemma \ref{lem:unique_small}, there exists a sufficiently small constant $\lambda_0 > 0$ such that $v(\lambda) = \tilde{v}(\lambda)$ for $\lambda \in (0, \lambda_0]$. This is a contradiction and completes the proof for the case $\sigma \neq 0$.

\medskip
\textbf{Case $\sigma =0$.} Let $\lambda >0$ be given and suppose $v, \tilde{v} \in X_\alpha$ are two fixed points of $\Ts_\lambda=\Ts_{\lambda}^{(0)}$ such that $v \neq \tilde{v}$. By the implicit function theorem and the nondegeneracy result in Lemma \ref{lem:nondeg}, we can construct two continuous branches $v(\sigma), \tilde{v}(\sigma) \in C((-\eps, \eps); X_\alpha)$ with some $\eps > 0$ such that
$$
v(\sigma) = \Ts_{\lambda}^{(\sigma)}[v(\sigma)] \quad \mbox{and} \quad \tilde{v}(\sigma) = \Ts_{\lambda}^{(\sigma)}[\tilde{v}(\sigma)] \quad \mbox{for $\sigma \in (-\eps,\eps)$},
$$  
with $v(0) = v$ and $\tilde{v}(0) = \tilde{v}$. Now let $\sigma_n \neq 0$ be a sequence such that $\sigma_n \to 0$. By the uniqueness result shown above for $\sigma \neq 0$, we conclude that $v(\sigma_n) = \tilde{v}(\sigma_n)$ for all $n \geq 1$. On the other hand, since $v(\sigma_n) \to v$ and $\tilde{v}(\sigma_n) \to \tilde{v}$ in $X_\alpha$, we deduce that $v \neq \tilde{v}$ cannot hold. 

The proof of Lemma \ref{lem:unique} is now complete.
\end{proof}

\subsection{Proof of Theorem \ref{thm:unique}} Let  $s \in (\frac 1 2, 1)$ be given. Suppose $u, \tilde{u} \in L_s(\R)$ are two solutions of \eqref{eq:Gelf2}. By Theorem \ref{thm:reg_sym} and the symmetries given in \eqref{eq:symmetry}, we can assume that $u, \tilde{u}$ are even functions with $u(0)=\tilde{u}(0)=0$. From the discussion in Section \ref{sec:setup}, we deduce that $v := \sqrt{e^u}$ and $\tilde{v} := \sqrt{e^{\tilde{u}}}$ are both fixed points of the map $\Ts_{\lambda} : X_\alpha \to X_\alpha$ with parameter $\lambda=v(0)=\sqrt{e^{u(0)}} = \tilde{v}(0) = \sqrt{e^{\tilde{u}(0)}} = 1$. By Lemma \ref{lem:unique}, we deduce that $v \equiv \tilde{v}$ and hence $u \equiv \tilde{u}$.

In summary, we have shown that any solution $u \in L_s(\R)$ of \eqref{eq:Gelf2} satisfies $u=Q_s$ (after translation and scaling), where $Q_s \in L_s(\R)$ is the unique even solution of \eqref{eq:Gelf2} with $Q_s(0)=1$. The proof of Theorem \ref{thm:unique} is now complete. \hfill $\qed$

\section{Finite Morse Index and Nondegeneracy}

\subsection{Proof of Proposition \ref{prop:morse}}

Let $s \in (\frac 1 2, 1)$ and suppose that $u \in L_s(\R)$ solves $\Ds u = e^u$ with $e^u \in L^1(\R)$. Let $\LL = \Ds - e^u$ denote the linearized operator viewed as a quadratic form on $C^\infty_c(\R)$. Recalling Theorem \ref{thm:non_existence}, we know that its Morse index satisfies $n_-(\LL) \geq 1$.

We claim that $n_-(\LL) < +\infty$ holds. [In fact, we obtain a quantitative bound below.] From the decay estimate in Lemma \ref{lem:decay} and the fact that $u$ is continuous (and thus locally bounded), it follows that $|u(x)| \leq C e^{-a|x|^{2 \alpha}}$ for all $x \in \R$ with some constants $a> 0$ and $C > 0$. Let us define $V := e^{u} \in L^\infty(\R) \cap L^1(\R; |x|^{2 \alpha} \, dx)$ and consider the self-adjoint Schr\"odinger operator
$$
H := \Ds - V    
$$
acting on $L^2(\R)$ with operator domain $H^{2s}(\R)$. From \cite{BrFaGr-24}[Theorem 1.1] we deduce that its number of strictly negative eigenvalues $N_{<0}(H)$ satisfies the bound
$$
N_{<0}(H) \leq C_s ( \| |x|^{2 \alpha} V \|_{L^1} + 1) < +\infty
$$
with some constant $C_s > 0$. We claim that the Morse index satisfies 
$$n_-(\LL) \leq N_{<0}(H).$$
[In fact, we can show that we must have equality here, but we only need the upper bound.] We argue by contradiction and suppose that $n_-(\LL) >N_{<0}(H)$ holds. Thus there exist linearly independent functions $\phi_1, \ldots, \phi_{N+1} \in C^\infty_c(\R)$ with  $N = N_{>0}(H)$ such that $\LL < 0$ is negative definite on $V = \mathrm{span} \{ \phi_1, \ldots, \phi_{N+1} \}$. But by the min-max principle, this implies that $H$ has at least $N+1$ negative eigenvalues. But this is a contradiction.

Next, we show that the fact that $n_-(\LL)$ is finite implies that $u$ is stable outside some compact subset $\KK \subset \R$.  Indeed, let $n= n_-(\LL) \geq 1$  and take $V = \mathrm{span} \{ \phi_1, \ldots, \phi_n \} \subset \CC^\infty_c(\R)$ such that $\langle \phi, \LL \phi \rangle < 0$ for all $\phi \in V \setminus \{ 0 \}$. By the maximality of $n=\dim V$, it directly follows that $\langle \psi, \LL \psi \rangle \geq 0$ for any $\psi \in \CC^\infty_c(\R \setminus \KK)$ with the compact set $\KK = \bigcup_{i=1}^n \mathrm{supp}(\phi_i)$.

Finally, we assume that $u \in L_s(\R)$ solves $\Ds u = e^u$ under the weaker assumption $e^u \in \Lloc^1(\R)$. Suppose that $u$ is stable outside some compact set $\KK \subset \R$. Using the Gagliardo seminorm 
$$
[\phi]_{\dot{H}^s}^2 = \frac{C_s}{2} \int_{\R} \! \int_\R \frac{|\phi(x)-\phi(y)|^2}{|x-y|^{1+2s} } \, dx \, dy  = \int_\R \phi \Ds \phi \, dx ,
$$
 we have
$$
[\phi]_{\dot{H}^s}^2 \geq \int_{\R} e^u \phi^2 \, dx \quad \mbox{for all $\phi \in C^\infty_c(\R \setminus \KK)$}.
$$ 
Let $R>0$ now be sufficiently large such that $\KK \subseteq \ov{B}_R = \{ |x| \leq R \}$. Next, we take a  function $\chi : \R \to \R$ with $\chi' \in L^\infty(\R)$ such that $0 \leq \chi \leq 1$ with the properties $\chi(x) \equiv 0$  for $|x| \leq R$ and $\chi(x) \equiv 1$ for $|x| \geq 2 R$. Furthermore, let $\phi \in \CC^\infty_c(\R)$ with $0 \leq \phi \leq 1$ and $\phi(x) \equiv 1$ for $|x| \leq 1$ and $\phi(x) \equiv 0$ for $|x| \geq 2$. For $k \geq 1$, we set $\phi_k(x) := \phi(x/k)$ and we consider the sequence of test functions $\psi_k \in \CC^\infty_c(\R \setminus \ov{B}_R) \subseteq \CC^\infty_c(\R \setminus \KK)$ given by
$\psi_k(x) := \chi(x) \phi_k(x)$ for $k \geq 1$. Applying an elementary fractional Leibniz type estimate in Lemma \ref{lem:leibniz}, we find
$$
[\psi_k]_{\dot{H}^s} \leq C + [\phi_k]_{\dot{H}^s} = C + k^{\frac{1}{2}-s} [\phi]_{\dot{H}^s} \leq K
$$
with some constant $K>0$ independent of $k \geq 1$. By monotone convergence and the stability of $u$ outside the compact set $\KK \subset \ov{B}_R$, we deduce that
$$
\int_{|x| \geq 2R} \eu^u \, dx = \lim_{k \to \infty} \int_{|x| \geq 2R} \eu^u \psi_k^2 \, d x \leq \sup_{k \geq 1} \, [\psi_k]_{\dot{H}^s}^2 < +\infty \, .
$$ 
Since  $\int_{|x| \leq 2R} \eu^u \, dx < +\infty$ by assumption, we obtain that $\eu^u \in L^1(\R)$ holds.

This completes the proof of Proposition \ref{prop:morse}. \hfill $\qed$

\subsection{Proof of Theorem \ref{thm:nondeg} (Nondegeneracy)}

Let $s \in (\frac 1 2,1)$ and assume that $u \in L_s(\R)$ solves $\Ds u = e^u$ with $e^u \in L^1(\R)$. By translation and scaling symmetry, it suffices to consider the case
$$
u(x) = Q_s(x) 
$$ 
with the unique smooth, even and monotone-decreasing function $Q_s$ from Theorem \ref{thm:unique}. 

Let $\LL = \Ds - e^u$ denote the linearized operator acting on 
$$
W_s = \{ \psi \in L_s(\R) : \mbox{$\psi(x) = O(|x|^{2 \alpha})$ as $|x| \to \infty$} \}.
$$ 
We recall that $\LL T = \LL R = 0$ with the smooth functions
$$
R = \pt_x u, \quad T = x \pt_x u + 2s.
$$
Since $R= \pt_x u= -H_\alpha(v^2)$, we conclude that $R(x) = o(1)$ as $|x| \to \infty$ and hence $R \in W_s$. Furthermore, we see that $T(x) = o(x)$ as $|x| \to +\infty$. Applying Proposition \ref{prop:integral_form} with $\rho = e^u T \in L^1(\R)$, we see that $T(x) = O(|x|^{2 \alpha})$ as $|x| \to \infty$, which implies that $R \in W_s$.

It remains to show that 
$$
\mathrm{ker} \, \LL = \mathrm{span} \{ R, T \}.
$$
Indeed, suppose that $\psi \in W_s$  solves $\LL \psi = 0$, which we can write as 
$$
\Ds \psi = v^2 \psi
$$
with $v= \sqrt{e^u}$. We decompose $\psi = \psi_e + \psi_o$ into even and odd parts, where it is easy to see that $\psi_{\#} \in L_s(\R)$ with $\psi_{\#}(x) = O(|x|^{2 \alpha})$ as $|x| \to \infty$, where $\#$ either stands for $e$ or $o$. Since $v^2=e^{u}$ is even, we obtain that
$$
\Ds \psi_\# = v^2 \psi_\#.
$$ 
Note also that $v^2 \psi_\# \in L^1(\R) \cap L^\infty(\R)$ due to the pointwise bound $v^2(x) \leq C e^{-a |x|^{2 \alpha}}$ with some $a >0$. By Proposition \ref{prop:integral_form} (where $B=0$ holds due to the assumed growth for $\psi$) and Proposition \ref{prop:riesz_integral}, we deduce that $\psi_\#$ is $C^1$ and satisfies the equation
\be \label{eq:psi_nondeg}
\psi_\#(x) = -\int_0^x H_\alpha(v^2 \psi_\#)(y) \, dy + \psi_\#(0).
\ee 
We next discuss the cases $\# \in \{e, o \}$ separately as follows.

\medskip
\textbf{Even Case.} Let $\psi_e \in W_s$ be an even solution of \eqref{eq:psi_nondeg}. We claim that 
\be \label{eq:T_nondeg}
\psi_e = \gamma T
\ee
for some $\gamma \in \R$. Indeed, let us define $h = \psi_e - \gamma T$ with some $\gamma \in \R$. Since $T(0) = 2s \neq 0$, we can choose $\gamma$ such that $h(0)=0$ holds. By linearity of \eqref{eq:psi_nondeg}, we deduce $h$ solves $h(x) = -\int_0^x H_\alpha(v^2 h)(y) \, dy$. By multiplying this equation with $v(x)>0$, we see that the function $f:= v h$ with $v(x) > 0$ satisfies
$$
f(x) = -v(x) \int_0^x H_\alpha(v f)(y) \, dy.
$$ 
From the estimate $0 < v(x) \leq C e^{-a |x|^{2 \alpha}}$ with some $a > 0$, we readily check that $f \in X_\alpha$. But since $f(0)=0$, we deduce from the proof of Lemma \ref{lem:nondeg} that $f \equiv 0$. By the strict positivity of $v(x) > 0$, this implies $h \equiv 0$. This proves that \eqref{eq:T_nondeg} holds.  

Thus we have shown that any even solution $\psi_e \in W_s$ of \eqref{eq:psi_nondeg} must satisfy $\psi_e \in \mathrm{span} \{ T \}$.

\medskip
\textbf{Odd Case.} Assume that $\psi_o \in W_s$ is an odd solution of \eqref{eq:psi_nondeg}. We claim that
\be \label{eq:R_nondeg}
\psi_o \in \mathrm{span} \{ R \}.
\ee
Suppose that $\psi_o \not \equiv 0$, since the claim is trivially true. By odd symmetry, we have $\psi_o(0)=0$ and using that $v^2 \psi_0 \in L^1(\R) \cap L^\infty(\R)$, we use Proposition \ref{prop:riesz_integral} to conclude
$$
\psi_o(x) = -c_\alpha \int_\R |x-y|^{2 \alpha}  (v^2 \psi_o(y)) \,dy 
$$ 
in view of the fact that $\int_{\R} |y|^{2 \alpha} v^2(y) \psi_o(y) \, dy = 0$, since $v^2 \psi_o \in L^1(\R; |x|^{2 \alpha} \, dx)$ is an odd function (recall that $v^2$ is even and the decay estimate $|v(x)| \leq C e^{-d|x|^{2 \alpha}}$). Using once again the $\psi_0$ is odd, we can rewrite the integral equation as
$$
\psi_0(x) = c_\alpha \int_0^\infty \left (|x+y|^{2 \alpha} - |x-y|^{2 \alpha} \right ) v(y)^2 \psi_o(y) \, dy.
$$
Multiplying this equation with $v(x) > 0$ and introducing the function $g(x) := v(x) \psi_o(x)$, we deduce
\be \label{eq:perron1}
g(x) = \int_0^\infty a(x,y) g(y) \, dy \quad \mbox{for $x > 0$},
\ee
with the positive kernel
$$
a(x,y) := c_\alpha v(x) \left (|x+y|^{2 \alpha} - |x-y|^{2 \alpha} \right ) v(y) > 0 \quad \mbox{for $x,y > 0$}.
$$
Thanks to the bound $|v(x)| \leq C e^{-a |x|^{2 \alpha}}$, we readily check that $a \in L^2(\R_{+} \times \R_{+})$. Thus the linear operator $A : L^2(\R_+) \to L^2(\R_+)$ with $A \phi(x) := \int_0^\infty a(x,y) \phi(y) \, dy$ is Hilbert--Schmidt and hence compact. Note also that $A=A^*$ is self-adjoint, since $a(x,y)=a(y,x)$ for the real-valued kernel $a(x,y)$. 

Since $g = v \psi_0 \in L^2(\R_+)$, we find from \eqref{eq:perron1} that $Ag = g$ holds, i.\,e., $g$ is an eigenfunction of $A$ with eigenvalue $1$. We claim that
\be \label{eq:span_h}
g \in \mathrm{span} \{ h \} \quad \mbox{with $h:=-v R$}.
\ee
Note that the strict positivity $h(x) := -v(x) R(x) > 0$ for $x >0$, since $v(x) > 0$ and $R(x)=\pt_x u=-H_\alpha(v^2) < 0$ for $x >0$ by Lemma \ref{lem:H_a} and the fact that $v=v^*$ is symmetric-decreasing. From $\LL R=0$ we conclude that $Ah = h$ by the discussion above.

Next, we claim that
\be
\mbox{$r=1$ is the largest eigenvalue of $A$ and it is simple}.
\ee
Indeed, let $r = \max \sigma(A)$ be the largest eigenvalue of the compact self-adjoint operator $A:L^2(\R_+) \to L^2(\R_+)$. By applying the Perron--Frobenius type result in Lemma \ref{lem:perron}, we deduce that $r >0$ exists and it is simple and its corresponding eigenfunction satisfies $\pm \psi(x) > 0$ for a.\,e.~$x > 0$. Since $Ah = h$, we must have that $r \geq 1$. Suppose now that $r > 1$ holds. Since $A=A^*$, we see that $A \psi = r \psi$ and $Ah = h$ implies that $(\psi, h) = 0$. But since $\pm \psi(x) > 0$ and $h(x) > 0$ for a.\,e.~$x>0$, this is impossible. Hence we can conclude that $r=1$ is the largest eigenvalue of $A$ and it is simple. In view of $Ag = g$ and $Ah = h$, we deduce  that \eqref{eq:span_h} holds. Thus we see that $g = \gamma h$ for some $\gamma \in \R$. Since $g= v \psi_o$ and $h=-v R$ together with $v(x) > 0$, this implies that $\psi_o = -\gamma R$, which shows that \eqref{eq:R_nondeg}. 

The proof of Theorem \ref{thm:nondeg} is now complete. \hfill $\qed$

\section{Generalizations to Nonconstant $K(x) > 0$}
\label{sec:thm_big}
In this section, we provide the proof of Theorem \ref{thm:big}, which considers solutions $u \in L_s(\R)$ of 
\be \label{eq:Gelf_K}
\Ds u = K e^u \quad \mbox{in $\R$}
\ee
with $s \in (\frac 1 2,1)$ and subject to the finiteness condition $\int_\R K e^u \,dx < +\infty$. Throughout the following, we assume that the prescribed function $K : \R \to \R_{>0}$ satisfies Assumption \textbf{(A)}. Moreover, we assume that $K(x) \not \equiv \mbox{const}.$ holds, since the case of constant $K(x)$ has already been covered by the previous discussion.

\subsection{Proof of Theorem \ref{thm:big}} We divide the proof into the following steps.

\medskip
\textbf{(i) Existence.} Let $\lambda > 0$ be given and consider the map $\Ts_\lambda : X_\alpha \setminus \{ 0 \} \to X_\alpha$ defined in \eqref{def:Ts}. By Corollary \ref{cor:exist}, there exists a fixed point $v \in X_\alpha \setminus \{ 0 \}$ of $\Ts_\lambda$. We see that $v(x) > 0$ is a $C^1$-solution of 
$$
\pt_x v = \frac{1}{2} (\pt_x \log K) v -\frac{1}{2} H_\alpha(v^2) v
$$
with $v(0) = \lambda \sqrt{K(0)} >0$. From applying Lemma \ref{lem:reformulation} and Proposition \ref{prop:integral_form}, we obtain that the even function $u := \log(K^{-1} v^2) \in L_s(\R)$ solves $\Ds u = K e^u$ with $K e^u = v^2 \in L^1(\R)$.

\medskip
\textbf{(ii) Regularity.} Let $u \in L_s(\R)$ solve \eqref{eq:Gelf_K}. By trivially adapting the proof of Theorem \ref{thm:reg_sym}, we find that $u \in \Cloc^{\gamma}(\R)$ with $\gamma = 2 \alpha =2 s-1 > 0$. Since $K$ is $C^1$, this implies that $K e^u \in \Cloc^\gamma(\R)$. From Lemma \ref{lem:regularity} it follows that $u \in \Cloc^{\gamma + 2s}(\R)$. 

\medskip
\textbf{(iii) Symmetry.} Assume that $u \in L_s(\R)$ solves \eqref{eq:Gelf_K} with $u(x) = O(|x|^{2 \alpha})$ as $|x| \to \infty$. From Lemma \ref{lem:symmetry} we deduce that
$u$ is an even function that is monotone-decreasing in $|x|$.

\medskip
\textbf{(iv) Uniqueness.} Next, we show that even solutions $u \in L_s(\R)$ of \eqref{eq:Gelf_K} are uniquely determined by its initial value $u(0)$ at the center of symmetry $x=0$. We recall that every even solution $u \in L_s(\R)$ must be integral by Proposition \ref{prop:integral_form}, i.\,e., we have that $B=0$ holds there and consequently we havet $u(x) = O(|x|^{2 \alpha})$ as $|x| \to \infty$.

Suppose now that $u \in L_s(\R)$ is an even solution of \eqref{eq:Gelf_K}. From the results in Section \ref{sec:setup} we deduce that $v := K e^u \in X_\alpha$ is a fixed point of the map $\Ts_{\lambda}^{(0)}= \Ts_{\lambda} : X_\alpha \setminus \{ 0 \} \to X_\alpha$ with $\lambda = v(0) = \sqrt{K(0) e^{u(0)}}$. By Lemma \ref{lem:unique}, we have uniqueness fixed points of $\Ts_\lambda$. Hence the claimed uniqueness result follows.

\medskip
\textbf{(v) Finite Morse Index.} Suppose that $u \in L_s(\R)$ is an even solution of \eqref{eq:Gelf_K} and hence $u(x) = O(|x|^{2 \alpha})$ as $|x| \to \infty$. Following exactly the proof of Proposition \ref{prop:morse}, we infer that the Schr\"odinger-type operator $H=\Ds-V$ with $V = K e^u$ has finitely many negative eigenvalues bounded by
$$
N_{<0}(H) \leq C_s ( \| |x|^{2 \alpha} V \|_{L^1} + 1) < +\infty,
$$
where we conclude that $|x|^{2 \alpha} K e^u \in L^1(\R)$ due to the bounds $\| K \|_{L^\infty} \leq K(0)$ and $e^{u(x)} \leq C e^{-d |x|^{2 \alpha}}$ with some constants $C, d > 0$. Using hat $N_{<0}$ is finite, we readily deduce that the Morse index $n_-(\LL) < +\infty$ is finite for the linearized operator $\LL = \Ds-K e^u$ viewed as a quadratic form on $C^\infty_c(\R)$; see the proof of Proposition \ref{prop:morse} for the full details. 

This completes the proof of Theorem \ref{thm:big}. \hfill $\qed$

\begin{appendix}

\section{Auxiliary Results} \label{sec:aux}

We collect some technical results needed in the paper.

\subsection{Fractional Leibniz Estimate}

\begin{lem} \label{lem:leibniz}
Let $R>0$ be given and suppose that $\chi \in  W^{1,\infty}(\R)$ with $0 \leq \chi \leq 1$ such that $\chi(x) = 0$ for $|x| \leq R$ and $\chi(x) = 1$ for $|x| \geq 2R$. Then, for any $s \in (0,1)$, we have
$$
[\chi \phi]_{\dot{H}^s} \leq C \left ( \|\phi\|_{L^\infty} + [\phi]_{\dot{H}^s}   \right ) \quad \mbox{for all $\phi \in \CC^\infty_c(\R)$} \, ,
$$
with some constant $C=C(\| \chi' \|_{L^\infty}, R, s) > 0$.
\end{lem}

\begin{proof}
This follows from a fractional Leibniz estimate, which can be deduced by elementary means. Indeed, we note that
\begin{align*}
[\chi \phi ]^2_{\dot{H}^s} & = \frac{c_s}{2} \int_\R \! \int_\R \frac{|\chi(x) \phi(x)- \chi(y) \phi(y)|^2}{|x-y|^{1+2s}} \, \df x \, \df y \\
& \leq C  \left (  \| \phi \|_{L^\infty}^2  \int_\R \! \int_\R \frac{|\chi(x)-\chi(y)|^2}{|x-y|^{1+2s}} \, \df x \, \df x + \int_\R \! \int_\R \frac{|\phi(x)-\phi(y)|^2}{|x-y|^{1+2s}} \, \df x \, \df y \right ) \\
& \leq C \left ( \| \phi \|_{L^\infty}^2 [\chi]_{\dot{H}^s}^2 + [\phi]_{\dot{H}^s}^2 \right ) \, , 
\end{align*}
which follows from the elementary inequality $|ab-cd|^2 \leq |a-c|^2 (|b|^2 + |d|^2) + |b-d|^2(|a|^2 + |c|^2)$ together with the uniform bounds $|\chi(x)| \leq 1$ and $|\phi(x)| \leq \| \phi \|_{L^\infty}$. It remains to show that
$$
[\chi]_{\dot{H}^s}^2 \leq K
$$
with some constant $K>0$ depending only on $\| \chi' \|_{L^\infty}, R$ and $s$. To this end, we note that $\chi(x)-\chi(y) = 0$ on the sets $\{ (x,y) \in \R^2 : |x| \leq R, |y| \leq R\}$ and $\{ (x,y) \in \R^2 : |x| \geq 2R, |y| \geq 2 R\}$. Using this fact together with the symmetry in $x$ and $y$, we conclude
\begin{align*}
[\chi]_{\dot{H}^s}^2 & \leq C \int_{|x| \leq 2 R} \int_{|y| \geq R} \frac{|\chi(x)-\chi(y)|^2}{|x-y|^{1+2s}} \, \df x \, \df y \\
& = \int_{|x| \leq 2R} \int_{|y| \geq 2R} ( \mathds{1}_{\{ |x-y| \leq 1\} } + \mathds{1}_{\{ |x-y| > 1\} } ) \frac{|\chi(x)-\chi(y)|^2}{|x-y|^{1+2s}} \, \df x \, \df y \\
& \leq C \|\chi'\|_{L^\infty}^2 \int_{|x| \leq 2 R} \int_{|x-y| \leq 1} |x-y|^{1-2s} \, \df x \, \df y  \\
& \quad + C \int_{|x| \leq 2R} \int_{|x-y| > 1} |x-y|^{-1-2s} \, \df x \, \df y \leq C R (\| \chi' \|_{L^\infty}^2 +1) \, .
\end{align*}
This completes the proof.
\end{proof}

\subsection{Compactness Results}

We have the following compactness result concerning subsets of the Banach space $X_\alpha$ with parameter $\alpha >0$ defined in Section \ref{sec:setup} above.
 
\begin{lem} \label{lem:compact}
Let $C > 0$ and $d >0$ be constants. Suppose that $\mathcal{F} \subset X_\alpha$ satisfies
$$
|v(x)| \leq C \eu^{-d |x|^{2 \alpha}} \quad |\pt_x v(x)| \leq C \quad \mbox{for all $v \in \mathcal{F}$ and a.\,e.~$x \in \R$} \, .
$$
Then $\mathcal{F}$ is relatively compact in $X_\alpha$. 
\end{lem}

\begin{proof}
Let $(v_k)_{k=1}^\infty$ be a sequence in $\mathcal{F}$. Thanks to the uniform Lipschitz bound on $(v_k)_{k=1}^\infty$ and by Arzel\`a--Ascoli, we have (after passing to a subsequence if necessary) that $v_k \to v$ in $\Lloc^\infty(\R)$ with some $v : \R \to \R$ such that $|v(x)| \leq C \eu^{-d |x|^{2 \alpha}}$. Since $v_n \to v$ in $\Lloc^\infty(\R)$ together with the uniform pointwise bound $|v_k(x)| \leq C \eu^{-d |x|^{2 \alpha}}$, we easily see that $v_k \to v$ in $L^2(\R, (1+|x|^{2 \alpha}) \, \df x) \cap L^\infty(\R)$.  Since all $v_k$ are even functions, so is its limit $v$. This proves that $\mathcal{F}$ is relatively compact in $X_\alpha$.
\end{proof}

\begin{lem} \label{lem:helly}
Let $1 \leq p_1 < p_2 < \infty$ be given. Suppose that $f_n=f_n^*$ is a sequence of symmetric-decreasing functions satisying
$$
\| f_n \|_{L^p} \leq C \quad \mbox{for all $p \in [p_1, p_2]$}
$$ 
with some constant $C>0$. Then there exists a subsequence $(f_{n_k})_{k=1}^\infty$ and a symmetric-decreasing function $f=f^*$ such that $f_{n_k} \to f$ almost everywhere in $\R$ and
$$
\mbox{$f_{n_k} \to f$ in $L^p(\R)$ for all $p \in (p_1, p_2)$}.
$$
\end{lem}

\begin{proof}
Using that $f_n=f_n^*$ are even functions, we write $f_n = f_n(r)$ with $r=|x| \geq 0$. Since each $f_n$ is monotone decreasing in $r$ and by the assumed uniform bound, we find for $p \in \{ p_1, p_2\}$ that $C \geq \int_{0}^r |f(s)|^p \, ds \geq r |f(r)|^p$ for almost every $r > 0$. Hence we have uniform pointwise bound
\be \label{ineq:helly}
|f_n(r)| \leq C \min \{ r^{-1/p_1}, r^{-1/p_2} \} \quad \mbox{for a.\,e.~$r> 0$}.
\ee
Now, since $f_n=f_n(r)$ are monotone decreasing, we can invoke Helly's selection principle to deduce that there exists a symmetric-decreasing function $f=f(r)$ such that $f_{n_k} \to f$ almost everywhere for some subsequence $(f_{n_k})_{k=1}^\infty$. Since the right-hand side in \eqref{ineq:helly} belongs to $L^p(\R)$ for any $p \in (p_1,p_2)$, we can use the dominated convergence theorem to deduce that $f_{n_k} \to f$ in $L^p(\R)$ for $p \in (p_1, p_2)$.
\end{proof}

\begin{lem} \label{lem:compact_frechet}
Let $X$ be a Banach space and suppose $\Omega \subseteq X$ is open. If $F : \Omega \to X $ is a compact map whose Fr\'echet derivative $DF(v)$ exists at some $v \in \Omega$, then $DF(v) \in \mathcal{L}(X)$ is a compact linear operator. 
\end{lem}

\begin{proof}
Suppose that the Fr\'echet derivative $DF(v)$ exists, but fails to be a compact linear map. Thus there exist $\eps_0 > 0$ and some sequence $y_n \in X$ with $\| y_n \| \leq 1$ such that
$$
\| DF(v) ( y_k -  y_\ell) \| \geq \eps_0 \quad \mbox{for $k \neq \ell$}.
$$ 
By the definition of Fr\'echet differentiability there exists some $\delta > 0$ such that
$$
\| F(v+h) - F(v) - DF(v) h \| \leq \frac{\eps_0}{4} \| h \| \quad \mbox{for $\| h \| \leq \delta$}.
$$
Let $\tau > 0$ be sufficiently small such that $\| \tau y_k \| \leq \delta$ and $v + \tau y_k \in \Omega$ for all $k \geq 1$. By the triangle inequality, it follows
\begin{align*}
\| F(v+\tau y_k)-F(v+ \tau y_\ell)\| & \geq \| DF(v)(\tau y_k - \tau y_\ell)\| - \| F(v+ \tau y_k) - F(v) - DF(v) \tau y_k \| \\
& \quad - \| F(v+ \tau y_\ell) - F(v)- DF(v) \tau y_\ell \| \geq \frac{\eps_0}{2} \tau   
\end{align*}
for $k \neq \ell$. But this contradicts our assumption that $F : \Omega \to X$ is a compact map.
\end{proof}

\subsection{Fr\'echet Differentiability} 
Next, we study the Fr\'echet differentiability of the map $\Tss$ introduced in Section \ref{sec:setup} above. 

\begin{prop} \label{prop:frechet}
The map $\Tss : X_\alpha \setminus \{ 0 \} \to X_\alpha$ is Fr\'echet differentiable with
$$
(D_v \Tss[v]h)(x) = -\Tss[v](x) \int_0^x H_\alpha(vh)(y) \, dy \quad \mbox{for all $h \in X_\alpha$}.
$$
Moreover, the bounded linear map $D_v \Tss[v]: X_\alpha \to X_\alpha$ is compact and its operator norm depends continuously on $(\lambda, \sigma, v) \in (0, \infty) \times \R \times X_\alpha$.
\end{prop}

\begin{proof}
Let $\lambda >0$, $\sigma \in \R$, and $v \in X_\alpha \setminus \{ 0 \}$ be given. Take $0 < \eps < \| v \|_{L^2}$. For $h \in X_\alpha$ with $\| h \|_{X^\alpha} < \eps$, we need to show that
\be \label{eq:Frechet}
\Tss(v+h) - \Tss(v) = L h + o(\| h \|_{X^\alpha}),
\ee
where $L \in \mathcal{L}(X_\alpha)$ is the linear bounded operator given by
$$
(Lh)(x) =-\Tss[v](x) \int_0^x H_\alpha(vh)(y) \, dy.
$$ 
Indeed, for $h \in X_\alpha$ with $\| h \|_{X^\alpha} < \eps$, we define 
$$
w_h(x) := -\frac{1}{2} \int_0^x H_\alpha((v+h)^2)(y) \, dy = -d_\alpha \int_{\R} (|x-y|^{2 \alpha} - |y|^{2 \alpha}) (v+h)(y)^2 \, dy.
$$
By following the proof of Proposition \ref{prop:T_bounds} and using that $\| v+h \|_{L^2} \geq \| v\|_{L^2} -\eps > 0$, we obtain the pointwise bound
\be \label{ineq:w_h}
w_h(x) \leq -a |x|^{2 \alpha} + b \quad \mbox{for $x \in \R$ and $h \in B_\eps(0)$},
\ee 
with some positive constants $a, b > 0$. Let us now write
$$
w_h(x) = w_0(x) + g_h(x),
$$
where
$$
g_h(x) = -\int_0^x H_\alpha(vh)(y) \, dy - \frac{1}{2} \int_0^x H_\alpha(h^2)(y) \, dy.
$$
By Taylor expansion, we get
\be \label{eq:taylor}
e^{w_h(x)} = e^{w_0(x) + g_h(x)} = e^{w_0(x)} + e^{w_0(x)} g_h(x) + \frac{1}{2} e^{\xi_h(x)} g_h(x)^2
\ee
with some $\xi_h(x) \leq -a |x|^{2 \alpha} + b$ for all $x \in \R$ and $h \in B_\eps(0)$ thanks to \eqref{ineq:w_h}. Thus,
\begin{align*}
\| e^{\xi_h} g^2_h \|_{X^\alpha}  & \leq  \left \| e^{-a |x|^{2 \alpha} + b} g_h^2 \right \|_{L^2}  + \left \| |x|^{\alpha} e^{-a |x|^{2 \alpha} + b} g_h^2 \right \|_{L^2} + \left \| e^{-a |x|^{2 \alpha} + b} g_h^2 \right \|_{L^\infty} \\
& \leq C \left ( \| (1+|x|^\alpha) |x|^{4 \alpha} e^{-a |x|^{2 \alpha} + b} \|_{L^2} + \| |x|^{4 \alpha} e^{-a|x|^{2 \alpha} + d} \|_{L^\infty} \right )  \| h \|_{L^2}^2 = o(\| h \|_{X^\alpha}).
\end{align*}
Here we also used the uniform pointwise bound
$$
|g_h(x)| \leq C |x|^{2 \alpha} ( \| v h \|_{L^1} + \| h^2 \|_{L^1}) \leq C |x|^{2 \alpha} (\| v \|_{L^2} \|h \|_{L^2} + \| h \|_{L^2}^2) \leq C |x|^{2 \alpha} \|v \|_{L^2} \| h\|_{L^2} 
$$
noticing that $\| h \|_{L^2} < \eps < \|v \|_{L^2}$ in the last step. Next, we note that
$$
g_h(x) = -\int_0^x H_\alpha(vh)(y) \, dy + r_h(x) \quad \mbox{with} \quad r_h(x) = -\frac{1}{2} \int_0^x H_\alpha(h^2)(y) \, dy.
$$
From the pointwise bound $|r_h(x)| \leq C |x|^{2 \alpha} \| h \|_{L^2}^2$, we deduce
$$
\| e^{w_0} r_h \|_{X^\alpha} \leq C \| h \|_{L^2}^2 = o(\|h \|_{X^\alpha}).
$$
By multiplying \eqref{eq:taylor} with $\lambda \sqrt{K(x)} e^{-\frac{1}{2} \sigma^2 x^2} \in L^\infty(\R)$, we deduce that \eqref{eq:Frechet} holds true.

Finally, since the map $\Tss : X_\alpha \setminus \{ 0 \} \to X_\alpha$ is locally compact, it follows that its Fr\'echet derivative $D_v \Tss[v]$ is a compact operator; see Lemma \ref{lem:compact_frechet} for a short proof of this general fact. Also, the continuity 
$$
\| D_v \Tss[v] - D_v \Ts_{\lambda_0}^{(\sigma_0)}[v_0] \|_{\mathcal{L}(X_\alpha)} \to 0 \quad \mbox{as} \quad (\lambda, \sigma, v) \to (\lambda_0, \sigma_0, v_0)
$$
follows by straightforward arguments, which we omit here.
\end{proof}

\subsection{Injectivity of the Laplace Transform}

Recall that $\mathsf{L}w(t) = \int_0^\infty e^{-st} w(s) \,ds$ denotes the (one-sided) Laplace transform of a function $w : [0,\infty) \to \R$. We have the following well-known fact, whose proof we provide for the reader's convenience.
 
\begin{lem} \label{lem:Laplace}
Let $w : [0,\infty) \to \R$ be bounded and continuous. Then $\mathsf{L}w(t) = 0$ for all $t > 0$ if and only if $w(s) = 0$ for all $s \geq 0$.
\end{lem}

\begin{remark*}
The arguments below easily generalize to continuous $w$ that are of exponential type, i.\,e., we have $|w(s)| \leq M e^{c s}$ for all $s \geq 0$ with some constants $M, c \geq 0$.
\end{remark*}

\begin{proof}

Suppose that $w : [0,\infty) \to \R$ be bounded and continuous. Clearly, we have that $|\mathsf{L}w(t)| \leq t^{-1} \| w \|_{L^\infty(\R_+)}$ for $t > 0$ and it follows that $\mathsf{L}w(t)$ is continuous in $t > 0$ by dominated convergence. Assume now that $\mathsf{L}w(t) = 0$ for $t > 0$. In particular, we can take $t=2 + n>0$ for any integer $n \geq 0$, which yields
$$
0 = \int_0^\infty e^{-2s-ns} w(s) \, ds = \int_0^1 y^{n+1} w(- \! \log y) \, dy = \int_0^1 y^n v(y) \, dy,
$$
using the substitution $y=e^{-s}$ and setting $v(y) := y w(- \! \log y)$ for $y \in (0,1]$. Note that $v$ is continuous on $(0,1]$ and it extends continuously to $y=0$ by setting $v(0)=0$, since $|v(y)| \leq y |w(- \! \log y)| \leq y \| w \|_{L^\infty(\R_+)}  \to 0$ as $y \to 0^+$. Hence we deduce that
\be \label{eq:moments}
\int_0^1 y^n v(y) \, dy = 0 \quad \mbox{for $n=0,1,2,\ldots$}
\ee
for the function $v \in C^0([0,1])$. But this implies that $v(y) \equiv 0$. Indeed, let $\eps > 0$ be given and choose a polynomial $p(x)$ such that $\|v-p\|_{L^\infty([0,1])} \leq \eps$ thanks to the Weierstrass approximation theorem. Observe 
\begin{align*}
\int_0^1 v(y)^2 \, dy & = \int_0^1(v(y)-p(y))v(y) \, dy + \underbrace{\int_0^1 p(y) v(y) \, dy}_{\displaystyle =0 \ \ \mbox{by \eqref{eq:moments}}} \\
& \leq \| v-p \|_{L^\infty([0,1])} \cdot \| v \|_{L^\infty([0,1])} \leq \eps \| v \|_{L^\infty([0,1])}.
\end{align*}
Since $\eps > 0$ is arbitrary, we conclude that $v \in C^0([0,1])$ must be $v \equiv 0$, which implies that $w(s) = e^{s} v(e^{-s}) = 0$ for all $s \geq 0$. 
\end{proof}

\subsection{Perron--Frobenius Type Result}

We have the following quite elementary version of a Perron--Frobenius type result, which is sufficient for our purposes to prove Theorem \ref{thm:nondeg}. Since readers might be less familiar with Perron--Frobenius theory (or the more general Krein--Rutnam theory), we provide the details needed in our context.
 
\begin{lem} \label{lem:perron}
Let $\Omega \subset \R^n$ be open and non-empty. Suppose that $A : L^2(\Omega) \to L^2(\Omega)$ is a compact operator with
$$
Af(x) = \int_{\Omega} a(x,y) f(y) \, dy
$$
such that $a(x,y) > 0$ for almost every $x,y \in \Omega$. Then $r = \max \sigma(A) >0$ is a simple eigenvalue of $A$. Moreover, every eigenfunction $\psi \in L^2(\Omega)$ with $A\psi = r \psi$ satisfies $\pm \psi(x) > 0$ for almost every $x \in \Omega$. 
\end{lem}

\begin{proof}
We set $\| \cdot \| = \| \cdot \|_{L^2(\Omega)}$ in what follows and define
$$
r := \sup_{\phi \in L^2, \| \phi \| = 1} (\phi, A \phi) .
$$
Note that $r > 0$ holds, which follows from taking some $\phi \in C^\infty_c(\Omega)$ with $\phi \geq 0$ and $\phi \not \equiv 0$, which implies that $(\phi, A \phi) > 0$ using that $a(x,y) > 0$ almost everywhere. This shows that $r > 0$ holds. Furthermore, since $A$ is compact it is easy to that the supremum above is attained. Let $\psi \in L^2(\Omega)$ be a maximizer. Since $A=A^*$ is self-adjoint, we find that $\psi$ is an eigenfunction with $A \psi = r \psi$. In particular, we see that $r > 0$ is an eigenvalue and it is easy to see that $r = \max \sigma(A)$ holds from the construction.

Suppose now that $\psi \in L^2(\Omega)$ is an eigenfunction with $A \psi = r \psi$. We claim that $|\psi| \geq 0$ is also an eigenfunction for the same eigenvalue $r$. Indeed, using that $a(x,y) > 0$ almost everywhere, we find
$$
 r= (\psi, A \psi) \leq  (|\psi|, A |\psi|) \leq r \| |\psi| \|^2 = r \| \psi \|^2 = r. 
$$ 
Hence $|\psi|$ is also a maximizer and must satisfy $A|\psi| = r |\psi|$. From $|\psi| \geq 0$ with $|\psi| \not \equiv 0$ and $r>0$, this implies that $|\psi(x)| > 0$ for almost every $x \in \Omega$. Now we consider $\phi := |\psi| - \psi \geq 0$, which also satisfies $A \phi = r \phi$. If $\phi(x) = 0$ almost everywhere, then $\psi(x) = |\psi(x)| > 0$ almost everywhere. If $\phi \not \equiv 0$, we re-run the argument above to conclude that $|\phi(x)| = | |\psi(x)| - \psi(x)| > 0$ almost everywhere. But this implies that $\pm \psi(x) >0$ almost everywhere.

Finally, we prove that $r$ is a simple eigenvalue. To show this, let $\psi_1, \psi_2 \in L^2(\Omega)$ be two linearly independent eigenfunctions with $A \psi_k = r \psi_k$ for $k=1,2$. Thus we can find another eigenfunction $\phi = \psi_1 + \gamma \psi_2 \not \equiv 0$ by choosing $\gamma \in \R$ such that $(\phi, \psi_1)=0$. Since $A \phi = r\phi$, we deduce that $\pm \phi(x) > 0$ almost everywhere. Since also $\pm \psi_1(x) > 0$ almost everywhere by the previous discussion, we obtain that $\pm (\phi, \psi_1) > 0$. But this contradicts the orthogonality condition $(\phi, \psi_1) =0$. Hence the eigenvalue $r$ must be simple.
\end{proof}

\section{Moving Plane Argument}

\label{sec:moving_plane}

In this section, we provide the details of the moving plane method used to complete the proof of Lemma \ref{lem:symmetry} above. We adapt the arguments in \cite{AhLe-22} and reference given there.

For $\lambda > 0$ and $x\in \R$, we set 
\begin{align*}
    &\Sigma_\lambda := [\lambda
    ,\infty), &  &x_\lambda := 2\lambda -x 
    , & &u_\lambda(x) := u(x_\lambda), & &\Sigma_\lambda^u := \{x\in \Sigma_\lambda : u(x) > u_\lambda(x) \}.
\end{align*}
We recall the integral representation \begin{align*}
    u(x) = \int_\mathbb{R} G(x-y) K(y) e^{u(y)} \, \mathrm{d}y + C \quad \mbox{with} \quad G(x) = -c_\alpha |x|^{2 \alpha}
\end{align*}
and some constant $C \in \mathbb{R}$. We initiate the moving plane method by showing that $\Sigma_\lambda^u$ is empty in the regime of sufficiently large $\lambda$.

\begin{prop}\label{prop:movingplane1}
    There exists $\lambda_0 > 0$ such that $\Sigma_\lambda^u = \emptyset$ for all $\lambda > \lambda_0$.
\end{prop}

\begin{proof}
    Let $\lambda > 0$. Using the integral representation and the fact that $G$ is an even function, a calculation yields \begin{align*}
        u(x)-u_\lambda(x) = \int_{\Sigma_\lambda} (G(x-y)-G(x_\lambda-y))\left( K(y) e^{u(y)} - K(y_\lambda) e^{u_\lambda(y)}\right) \, dy.
    \end{align*}
    Since $G$ and $K$ are even and monotone decreasing in $|x|$, we have
    $$
    0 \geq G(x-y) \geq G(x_\lambda-y) \quad \mbox{and} \quad 0 \leq K(y) \leq K(y_\lambda) \quad \mbox{for $x,y \in \Sigma_\lambda$}.
    $$
    For any $x,y \in \Sigma_\lambda^u \subset \Sigma_\lambda$, we thus estimate 
    \begin{align*}
        u(x) - u_\lambda(x) &\leq \int_{\Sigma_\lambda^u} (G(x-y)-G(x_\lambda-y)) \left(K(y) e^{u(y)}-K(y_\lambda) e^{u_\lambda(y)}\right) \, dy \\
        &\leq \int_{\Sigma_\lambda^u} (G(x-y)-G(x_\lambda-y)) F_\lambda(y) \, dy,
    \end{align*}
    where we denote \begin{align*}
        F_\lambda(y) := K(y) e^{u(y)} (u(y)-u_\lambda(y)).
    \end{align*}
    Note that $F_\lambda(y) \geq 0$ for $y \in \Sigma_\lambda^u$. Next, we observe the upper bound 
    $$
        -G(x_\lambda-y) = c_\alpha |2\lambda -x -y|^{2\alpha} \leq c_\alpha |x + y|^{2 \alpha} \leq c_\alpha \left(|x|^{2 \alpha} + |y|^{2 \alpha}\right),
    $$
    for $x,y \in \Sigma_\lambda$. Thus, for $x \in \Sigma_\lambda^u$, 
    \begin{align*}
        0 < u(x) - u_\lambda(x) \leq c_\alpha \int_{\Sigma_\lambda^u} \left(|x|^{2\alpha} + |y|^{2 \alpha}\right) F_\lambda(y) \, dy.
    \end{align*}
    Using the standard notation $\langle x \rangle = \sqrt{1+x^2}$, we deduce that
    \begin{align*}
        \left|\left|\langle x\rangle^{-2}(u-u_\lambda)\right|\right|_{L^1(\Sigma_\lambda^u)} &\leq c_\alpha \int_{\Sigma_\lambda^u} \int_{\Sigma_\lambda^u} \langle x\rangle^{-2} \left(|x|^{2 \alpha} + |y|^{2\alpha}\right) F_\lambda(y) \, \mathrm{d}y \mathrm{d}x \\
        &\leq C(\alpha) \int_{\Sigma_\lambda^u} \left(1 + |y|^{2 \alpha}\right) F_\lambda(y) \, \mathrm{d}y
    \end{align*}
    with some constant $C(\alpha)>0$ independent of $\lambda >0$ and using that $\int_{\R} \langle x \rangle^{-2} |x|^{2 \alpha} \, dx < +\infty$ for $\alpha \in (0, \frac{1}{2})$. Therefore, we have found that 
    \begin{align*}
        &\left \| \langle x\rangle^{-2}(u-u_\lambda)\right \|_{L^1(\Sigma_\lambda^u)} \leq C(\alpha) \int_{\Sigma_\lambda^u} \left(1 + |y|^{2 \alpha}\right) K(y) e^{u(y)}(u(y)-u_\lambda(y)) \, dy \\
        &\leq C(\alpha) \sup_{y\geq \lambda} \left(\left(1 + |y|^{2 \alpha}\right) \langle y\rangle^2 K(y) e^{u(y)}\right) \left \|\langle x\rangle^{-2}(u-u_\lambda)\right \|_{L^1(\Sigma_\lambda^u)},
    \end{align*}
    where the constant $C(\alpha)>0$ is independent of $\lambda$. Recalling the decay estimae $0 \leq K e^u \leq C e^{-a |x|^{2 \alpha}}$ with some $a > 0$, we can find some constant $\lambda_0 > 0$ sufficiently large such that \begin{align*}
        \left|\left|\langle x\rangle^{-2}(u-u_\lambda)\right|\right|_{L^1(\Sigma_\lambda^u)} \leq \frac{1}{2} \left|\left|\langle x\rangle^{-2}(u-u_\lambda)\right|\right|_{L^1(\Sigma_\lambda^u)},
    \end{align*}
    for $\lambda > \lambda_0$. Hence the set $\Sigma^u_\lambda$ has measure zero for $\lambda > \lambda_0$, which implies that $\Sigma_\lambda^u= \emptyset$ for $\lambda> \lambda_0$ by continuity of $u-u_\lambda$.
    \end{proof}

As a next step, we establish the following continuation property.

\begin{prop} \label{prop:movingplane2}
    Suppose that $\lambda_0 > 0$ satisfies $\Sigma_\lambda^u = \emptyset$ for all $\lambda > \lambda_0$ and assume that $u \not\equiv u_{\lambda_0}$ on $\Sigma_{\lambda_0}$. Then there exists $\eps > 0$ such that $\Sigma_\lambda^u = \emptyset$ for all $\lambda > \lambda_0 - \eps$.
\end{prop}

\begin{proof}
    We divide the proof into the following steps. \\ \\
    \textbf{Step 1.} By assumption, we have $u(x) \leq u_\lambda(x)$ for all $x\geq \lambda$ and $\lambda > \lambda_0$. By continuity, we conclude that $u(x) \leq u_{\lambda_0}(x)$ for all $x\geq \lambda_0$. This shows that $\Sigma_{\lambda_0}^u = \emptyset$ holds. To show that we indeed have $\Sigma_\lambda^u = \emptyset$ for any $\lambda > \lambda_0-\eps$ with some $\eps > 0$, we argue as follows. First, we claim that the strict inequality holds: \begin{align}\label{eq:symAsLemProof}
        u(x) < u_{\lambda_0}(x) \quad \mbox{for all $x > \lambda_0$}.
    \end{align}
 We argue by contradiction. Suppose that $u(x) = u_{\lambda_0}(x)$ for some $x>\lambda_0$. Note that $\Ds u = K e^u$ and from our assumption that $K \in C^1$, we can deduce that $u \in \Cloc^{2s+\gamma}(\R)$ (see also the proof of Theorem \ref{thm:big} (iii) for details). Thus $\Ds u(x) = e^{u(x)}$ holds pointwise in $\R$. Since $K$ is monotone-decreasing in $|x|$, we find 
    \be  \label{eq:movingplane1}
        (-\Delta)^s (u_{\lambda_0}-u)(x) = (K(x_{\lambda_0}) - K(x)) e^{u(x)} \geq 0.
    \ee
    On the other hand, using that $u(x) = u_{\lambda_0}(x)$ for some $x > \lambda_0$, we conclude 
    \begin{align*}
        (-\Delta)^s(u_{\lambda_0}-u)(x)  &= -C_s \, \mathrm{PV} \int_\mathbb{R} \frac{u_{\lambda_0}(y)-u(y)}{|x-y|^{1+2s}} \, dy \\
        = -C_s \, \mathrm{PV} &\int_{\lambda_0}^\infty \left(\frac{1}{|x-y|^{1+2s}} - \frac{1}{|x-y_{\lambda_0}|^{1+2s}}\right)(u_{\lambda_0}-u)(y) \, dy \leq 0.
    \end{align*}
    In view of \eqref{eq:movingplane1}, we must have equality to zero above. Since $x>\lambda_0$, we deduce that $u \equiv u_{\lambda_0}$ on $\Sigma_{\lambda_0}$. This contradicts our assumption that $u \not \equiv u_{\lambda_0}$ on $\Sigma_{\lambda_0}$. This proves  \eqref{eq:symAsLemProof}. \\ \\
    \textbf{Step 2.} Similarly to the proof of Proposition \ref{prop:movingplane1}, we get the estimate 
    $$
        \left|\left|\langle x\rangle^{-2}(u-u_\lambda)\right|\right|_{L^1(\Sigma_\lambda^u)}\leq C(\lambda) \sup_{y\geq \lambda} \left(\left(1 + |y|^{2\alpha }\right) \langle y\rangle^2 K(y) e^{u(y)}\right) \left \| \langle x\rangle^{-2}(u-u_\lambda)\right \|_{L^1(\Sigma_\lambda^u)}.
    $$
    Now, in view of \eqref{eq:symAsLemProof}, we conclude that the set \begin{align*}
        \overline{\Sigma_{\lambda_0}^u} := \{x \in \Sigma_{\lambda_0} : u(x) \geq u_{\lambda_0}(x)\} = \{x= \lambda_0\}
    \end{align*}
    has Lebesgue measure zero. Since $\lim_{\lambda \nearrow \lambda_0} \Sigma_\lambda^u \subset \overline{\Sigma_{\lambda_0}^u}$, and by inspecting the expression for $C(\lambda)$ in the proof of Proposition \ref{prop:movingplane1}, we deduce that $C(\lambda) \rightarrow 0$ as $\lambda \nearrow \lambda_0$. Thus, for some $\eps> 0$ sufficiently small, we find that 
    \begin{align*}
        \left \| |\langle x\rangle^{-2}(u-u_\lambda)\right \|_{L^1(\Sigma_\lambda^u)} \leq \frac{1}{2} \left \|\langle x\rangle^{-2}(u-u_\lambda)\right \|_{L^1(\Sigma_\lambda^u)}
    \end{align*}
    for $\lambda > \lambda_0 - \eps$. It follows that $\Sigma_\lambda^u$ has measure zero and thus $\Sigma_\lambda^u = \emptyset$ by the continuity of $u-u_\lambda$. This completes the proof.
\end{proof}

We are now ready to complete the proof of the symmetry result stated in Lemma \ref{lem:symmetry} as follows. Let us consider 
\begin{equation*}
        \lambda_* := \inf\left\{\lambda>0 : \Sigma_{\lambda'}^u = \emptyset \text{ for all } \lambda' > \lambda \right\} \geq 0,
    \end{equation*}
     which is is well-defined thanks to Proposition \ref{prop:movingplane1}. Suppose now that $\lambda_* > 0$ holds. By Proposition \ref{prop:movingplane2}, we must have $u \equiv u_{\lambda_*}$ on $\Sigma_{\lambda_*}= [\lambda_*, \infty)$. But this implies that $u$ is even with respect $\lambda_*$. Indeed, let $x > 0$ be arbitrary. Since $\lambda_* + x \in  \Sigma_{\lambda_*}$ and $u \equiv u_{\lambda_*}$ on $\Sigma_{\lambda_*}$, we conclude $u(\lambda_*+x) = u_{\lambda_*}(\lambda_*+x) = u(\lambda_*-x)$ as claimed, and the later identity trivially extends to all $x \in \R$. 
     
     Assume now that $\lambda_* = 0$ holds. Since $K(x) = K(-x)$ is even, the function $\Tilde{u}(x) = u(-x)$ is also a solution of $(-\Delta)^s u = K e^u$. By re-running the argument above, we conclude that $\Tilde{u}$ is either symmetric around some $\lambda_*>0$ and thus, $u$ is symmetric around $-\lambda_*$, or we have $\lambda_*=0$ and hence $u$ is symmetric around $0$. In summary, we conclude that $u$ is symmetric around some point $x_* = \pm \lambda_* \in \mathbb{R}$, i.\,e., it holds
    $$
    u(x_0+x) = u(x_0-x) \quad \mbox{for all $x \in \R$}.
    $$  
    
    Next, we show that $u$ is symmetric-decreasing around $x_*$. Without loss of generality we assume that $y>x>x_*\geq 0$. We define $\lambda := \frac{x+y}{2}>x_*$, which implies that $x = 2\lambda - y = y_\lambda$. By construction, we have $\Sigma^u_\lambda = \emptyset$ and thus it follows
    \begin{equation*}
        u(y) \leq u_\lambda(y) = u(y_\lambda) = u(x).
    \end{equation*}
    This shows that $u$ is symmetric-decreasing with respect to $x_*$.
    
    Finally, we show that $x_* \neq 0$ implies that $K(x) \equiv \mbox{const}.$ holds. Indeed, we can assume $x_* = \lambda_* > 0$ without loss of generality and thus $u \equiv u_{\lambda_*}$ on $\Sigma_{\lambda_*}$. From the argument in the proof of Proposition \ref{prop:movingplane2} above, we deduce that
    $$
    0=\Ds( u_{\lambda_*} - u)(y) = (K(y_{\lambda_*}) - K(y)) e^{u(y)} \quad \mbox{for $y \in \Sigma_{\lambda_*}$}.
    $$
    Thus $K_{\lambda_*} \equiv K$ on $\Sigma_{\lambda_*}$, which means that $K$ is symmetric with respect to $\lambda_* > 0$, i.\,e.,
    $$
    K(\lambda_* - x) = K(\lambda_* + x) \quad \mbox{for $x \in \R$}.
    $$
    In particular, this implies that $K(0) = K(2 \lambda_*)$. On the other hand, since $K(x) = K(-x)$ is assumed to be even and monotone-decreasing in $|x|$, we find $K(\lambda_*) \leq K(0) = K(2 \lambda_*)$. Since $K$ is monotone-decreasing in $|x|$, this implies $K(x) \equiv K(0)$ on $[-2 \lambda_*, 2 \lambda_*]$. Next, we note that $K(3 \lambda_*) = K(-\lambda_*) = K(0)$. Likewise, we find that $K(x) \equiv K(0)$ on $[-3 \lambda_0, 3 \lambda_0]$. A simple iteration now yields $K(x) \equiv K(0)$ on $[-n \lambda_0, n \lambda_0]$ for any $n \in \N$, which implies that $K(x) \equiv K(0)$ on $\R$. \hfill $\qed$
    
\end{appendix}

\bibliographystyle{siam}
\bibliography{GelfandBib}

\end{document}